\documentclass[a4paper]{article}

\usepackage{a4wide}
\usepackage[latin1]{inputenc}
\usepackage[T1]{fontenc}
\usepackage{amsfonts}
\usepackage{amssymb}
\usepackage{amsmath}
\usepackage{amsthm}
\usepackage{mathrsfs}
\usepackage{bbm}
\usepackage{dsfont}
\usepackage{color}
\usepackage{graphicx}

\newcommand{\EE}{\mathds{E}}

\newcommand{\NN}{\mathds{N}}
\newcommand{\PP}{\mathds{P}}

\newcommand{\RR}{\mathds{R}}

\newtheorem{thm}{Theorem}[section]
\newtheorem{lemme}[thm]{Lemma}
\newtheorem{prop}[thm]{Proposition}
\newtheorem{cor}[thm]{Corollary}

\theoremstyle{definition}

\numberwithin{equation}{section}

\author{Hendrik Weber\\ Universit\"at Bonn}
\title{Sharp interface limit for invariant measures of a stochastic Allen-Cahn equation}

\begin{document}
\maketitle
\begin{abstract}
The invariant measure of a one-dimensional Allen-Cahn equation with an additive space-time white noise is studied. This measure is absolutely continuous with respect to a Brownian bridge with a density which can be interpreted as a potential energy term. We consider the sharp interface limit in this setup. In the right scaling this corresponds to a Gibbs type measure on a growing interval with decreasing temperature. Our main result is that in the limit we still see exponential convergence towards a curve of minimizers of the energy  if the interval does not grow too fast. In the original scaling the limit measure is concentrated on configurations with precisely one jump. This jump is distributed uniformly.
\end{abstract}

\vskip3ex


\vskip3ex

\section{Introduction}
Reaction-diffusion equations can be used to model phase separation and boundary evolutions in various physical contexts. Typically behavior of boundaries or geometric evolution laws are studied with the help of such equations. Often in such models one includes an extra noise term. This may happen for various reasons -- the noise may be a simplified model for effect of additional degrees of freedom that are not reflected in the reaction-diffusion equation. From a numerical point of view noise may improve stability in the simulations. In some systems there is even a justification for an extra noise term from a scaling limit of microscopic particle systems.

\vskip3ex

\emph{1. Setup and first main result}

\vskip1ex

The system considered here is the case of a symmetric bistable potential with two wells of equal depths. To be more precise, for a small parameter $\varepsilon >0$ we are interested in the equation
\begin{equation}\label{SACE}
\begin{split}
\partial_t u(x,t) &= \Delta u(x,t) - \varepsilon^{-1-\gamma}F'(u(x,t)) + \varepsilon^{(1-\gamma)/2} \partial_x \partial_t W(x,t)  \qquad (x,t) \in (-1,1) \times \RR_+\\ 
u(-1,t)&=-1 \qquad u(1,t)=1  \qquad  t \in \RR_+.
\end{split}
\end{equation}
Here $F$ is supposed to be a smooth (at least $C^3$) symmetric double-well potential i.e. we assume that $F$ satisfies the following properties:
\begin{equation}\label{assumptions}
\begin{cases}
(a) \quad F(u) \geq 0 \qquad \text{and} \quad F(u)=0 \quad \text{iff } u=\pm 1, \\
(b) \quad F' \text{ admits exactly three zeros $\{\pm 1,0 \}$ and $F''(0) < 0$, $F''(\pm 1)>0$},\\
(c) \quad F \text{ is symmetric, } \forall u \geq 0 \quad F(u)=F(-u).
\end{cases}
\end{equation}
A typical example is $F(u)=\frac{1}{2}(u^2-1)^2$. The expression $\partial_x \partial_t W(x,t)$ is a formal expression denoting space-time white noise. Such equation can be given rigorous sense in various ways, for example in the sense of mild solutions (\cite{Iw87,dPZ92}) or using Dirichlet forms \cite{AR90}. We are interested in the behavior of the system in the sharp interface limit $\varepsilon \downarrow 0$. The parameter $\gamma>0$ is a scaling factor. Our result will be valid for $\gamma < \frac{2}{3}$.  

\vskip3ex

We study the behavior of the invariant measure of (\ref{SACE}). This measure can be described quite explicitly as follows (\cite{dPZ96,VR05}): Let $\tilde{\nu}^\varepsilon$ be the law of a rescaled Brownian bridge on $[-1,1]$ with boundary points $\pm 1$. More precisely $\tilde{\nu}^\varepsilon$ is the law of a Gaussian process $(\tilde{u}(s),s \in [-1,1])$ with expectations $\EE \left[ \tilde{u}(s)\right]=s \quad \forall s \in [-1,1]$ and covariance $\text{Cov}(\tilde{u}(s),\tilde{u}(s'))= \varepsilon^{1-\gamma} \bigl(s \wedge s'+1 -\frac{(s+1)(s'+1)}{2} \bigr)$. Another equivalent way to characterize $\tilde{\nu}^\varepsilon$ is to say that it is a Gaussian measure on $L^2[-1,1]$ with expectation function $s \mapsto s$ and covariance operator $\varepsilon^{1-\gamma} (-\Delta)^{-1}$ where $\Delta$ denotes the one-dimensional Dirichlet Laplacian. Even another equivalent way is to say that $\tilde{u}(s)$ is the solution to the stochastic differential equation (SDE)
\[
 \mathrm{d}\tilde{u}(s)= \varepsilon^{\frac{1-\gamma}{2}} \mathrm{d}B(s) \qquad \tilde{u}(-1)=-1
\]
with some Brownian motion $B(s)$ conditionned on $\tilde{u}(1)=1$. Then the invariant measure $\tilde{\mu}^\varepsilon$ of (\ref{SACE}) is absolutely continuous with respect to $\tilde{\nu}^\varepsilon$ and is given as
\begin{equation}\label{InvMeas}
 \tilde{\mu}^\varepsilon( \mathrm{d}\tilde{u})= \frac{1}{Z^\varepsilon} \exp \Bigl(-\frac{1}{\varepsilon^{1+\gamma}} \int_{-1}^1 F(\tilde{u}(s)) \, \mathrm{ds} \Bigr) \tilde{\nu}^\varepsilon ( \mathrm{d}\tilde{u}). 
\end{equation}
Here $Z^\varepsilon= \int \exp \Bigl(- \frac{1}{\varepsilon^{1+\gamma}} \int_{-1}^1 F(\tilde{u}(s)) \,  \mathrm{ds} \Bigr) \tilde{\nu}^\varepsilon ( \mathrm{d}\tilde{u})$ is the appropriate normalization constant. The first main result of this work is the following:
\begin{thm}\label{THM1}
 Assume $0<\gamma < \frac{2}{3}$. Then the measures $\tilde{\mu}^\varepsilon( \mathrm{d}u)$ converge weakly for $\varepsilon \downarrow 0$ as measures on $L^2[-1,1]$ towards a limit measure $\tilde{\mu}$. This measure $\tilde{\mu}$ can be described as follows: If $\tilde{u} \sim \tilde{\mu}$ is a random function distributed according to $\tilde{\mu}$, then $\tilde{u}$ can almost surely be written as 
\[
 \tilde{u}(s)= -\mathbf{1}_{[-1,\xi[}+ \mathbf{1}_{[\xi,1]},
\]
where $\xi$ is random, uniformly distributed in $[-1,1]$.
\end{thm}

\begin{figure}
\begin{center}
\includegraphics[height=6cm]{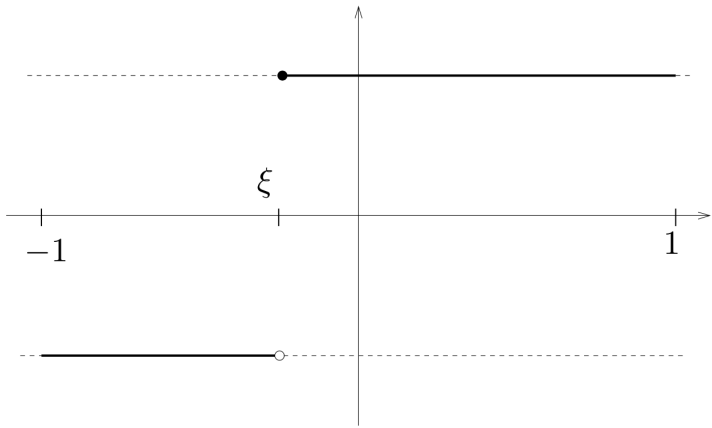}\label{ShIf}
\caption{The function  $-\mathbf{1}_{[-1,\xi[}+ \mathbf{1}_{[\xi,1]}$.}
\end{center} 
\end{figure}

Note that by Schilder's theorem together with an exponential tilting argument (such as \cite{dH00} Theorem III.17 on page 34), in the case where $\gamma=0$ the measures $\mu^\varepsilon$ concentrate exponentially fast around the unique minimizer of 
\[
 u \mapsto \int_{-1}^1 \left[ \frac{|u'(s)|^2}{2}+F\big(u(s)\big)  \right] \, \mathrm{ds},
\]
under the appropriate boundary conditions. In particular the weak limit is a Dirac measure on this minimizer. Furthermore the minimizer is not a step function.

\vskip3ex

One can remark that by an application of Girsanov's theorem also the measure $\tilde{\mu}^\varepsilon$ can be considered as distribution of the solution of a SDE which is conditioned on the right boundary values (see \cite{RY99} Chapter VIII  \textsection 3 and also \cite{HSV07,VR05}). It could be possible to obtain similar results by studying this SDE with help of large deviation theory (see for example \cite{S95}). We do not follow such an approach but conclude from Theorem \ref{THM2} which is obtained essentially by a discretization argument.

\vskip3ex

The reader might consider it unusual to work with $\tilde{\mu}^\varepsilon$ as measure on $L^2[-1,1]$ instead of $C[-1,1]$ or the space of c\`adl\`ag functions $D[-1,1]$. But all the estimates are given in the Hilbert-space setting. Also the class of \emph{continuous} processes is closed under weak convergence of measures on $D[-1,1]$. So certainly no similar result can be expected on this space. 

\vskip3ex

\emph{2. Feynman Heuristic and second main result}

\vskip1ex

Often important intuition on a measure on path space can be gained from considering Feynman's heuristic interpretation. In our context this heuristic interpretation states that $\tilde{\nu}^\varepsilon( \mathrm{d}\tilde{u})$ is proportional to a measure 
\[
 \exp \Bigl( -\frac{1}{\varepsilon^{1-\gamma}} \int_{-1}^1 \frac{|\tilde{u}'(s)|^2}{2} \, \mathrm{ds} \Bigr)  \mathrm{d} \tilde{u}   
\]
where $ \mathrm{d}\tilde{u}$ is a flat reference measure on path space. Of course this picture is non-rigorous: Such a measure $ \mathrm{d}\tilde{u}$ does not exist and the quantity $\int_{-1}^1 \frac{\tilde{u}'(s)^2}{2} \, \mathrm{ds}$ is almost surely not finite under $\tilde{\nu}^\varepsilon( \mathrm{d}\tilde{u})$. Nontheless it is rigorous on the level of finite dimensional distributions, and various classical statements about Brownian motion such as Schilder's theorem or Girsanov theorem have an interpretation in terms of this heuristic picture. The measure   $\tilde{\mu}^\varepsilon( \mathrm{d}\tilde{u})$ can then be interpreted as proportional to
\[
 \exp \Bigl( -\frac{1}{\varepsilon^{1 +\gamma}} \int F(\tilde{u}(s)) \, \mathrm{ds}-\frac{1}{\varepsilon^{1-\gamma}} \int_{-1}^1 \frac{|\tilde{u}'(s)|^2}{2} \, \mathrm{ds} \Bigr)  \mathrm{d}\tilde{u}.
\]
As one wants to observe an effect which results from the interaction of the \emph{potential term}  \linebreak $\frac{1}{\varepsilon^{1+\gamma}} \int F\big(\tilde{u}(s) \big)  \, \mathrm{ds}$  and the \emph{kinetic energy type term} $\frac{1}{\varepsilon^{1-\gamma}} \int_{-1}^1 \frac{\tilde{u}'(s)^2}{2}  \mathrm{ds}$ it seems reasonable to transform the system in a way that guarantees that these terms scale with the same power of $\varepsilon$. This transformation is given by stretching the random functions onto a growing interval  $[-\varepsilon^{-\gamma}, \varepsilon^{-\gamma}]$. More precisely consider the operators
\begin{equation*}
 \begin{split}
T^\varepsilon \colon L^2[-1,1] \rightarrow L^2[-\varepsilon^{-\gamma},\varepsilon^{-\gamma}] \qquad T^\varepsilon \tilde{u} (s)= \tilde{u}( \varepsilon^{\gamma}s).
   \end{split}
\end{equation*}
Then consider the pushforward measures $\mu^\varepsilon= T^\varepsilon_\#\tilde{\mu}^\varepsilon $. These measures are again absolutely continuous with respect to Gaussian measures: $\nu^ \varepsilon$ is the Gaussian measure on $L^2[-\varepsilon^{-\gamma},\varepsilon^{-\gamma}]$ with expectation function $s \mapsto \varepsilon^\gamma s$ and covariance operator $\varepsilon (-\Delta)^{-1}$. The other equivalent characterizations for $\tilde{\nu}^\varepsilon$ can be adapted with the right powers of $\varepsilon$. The measure $\mu^\varepsilon$ is then given as
\[
 \mu^\varepsilon( \mathrm{d}u)= \frac{1}{Z^\varepsilon} \exp \Bigl (-\varepsilon^{-1} \int_{-1}^1 F(u(s)) \, \mathrm{ds} \Bigr) \nu^\varepsilon ( \mathrm{d}u). 
\]
Note that the normalization constant $Z^\varepsilon$ is the same as above. In the Feynman picture this suggests that $\mu^\varepsilon( \mathrm{d}u)$ is proportional to  
\[
 \exp \Bigl(-\frac{1}{\varepsilon} \int_{-\varepsilon^{-\gamma}}^{\varepsilon^{-\gamma}} \left[ \frac{|u'(s)|^2}{2}+F\big( u(s) \big) \right]  \, \mathrm{ds} \Bigr) \mathrm{d}u. 
\]

\vskip3ex 

This motivates to study the energy functional appearing in the exponent: For functions $u\colon \RR \to \RR$ defined on the whole line with boundary conditions $u(\pm \infty)= \pm 1$ consider the energy functional 
\[
 \mathcal{H}(u)= \int_{-\infty }^{\infty } \left[ \frac{|u'(s)|^2}{2}+F(u(s)) \right] \, \mathrm{ds}-C_*.
\]
Here $C_*$ is a constant chosen in a way to guarantee that the minimizers of  $\mathcal{H}$ with the right boundary conditions verify $\mathcal{H}(u)=0$. This is the one-dimensional version of the well known real Ginzburg-Landau energy functional. There is a unique minimizer $m$ of $\mathcal{H}$ subject to the condition $m(0)=0$ and all the other minimizers are obtains via translation of $m$. More details on the energy functional and the minimizers can be found in Section \ref{SEC2}. Denote by $M$ the set of all these minimizers and by $m+L^2(\RR):=\{u \colon \RR \to \RR, u-m \in L^2(\RR) \}$ and $m+H^1(\RR):=\{u \colon \RR \to \RR, u-m \in H^1(\RR) \}$ the spaces of functions with the right boundary values. Note that every random function distributed according to  $\mu^\varepsilon( \mathrm{d}u)$ can be considered as function in $m+L^2(\RR)$ by trivial extension with $\pm 1$ outside of $[-\varepsilon^{-\gamma},\varepsilon^{-\gamma}]$. In this way $\mu^\varepsilon( \mathrm{d}u)$ can be interpreted as measure on $m+L^2(\RR)$.  We can now state the second main result of this work:

\begin{thm}\label{THM2}
Assume $0<\gamma < \frac{2}{3}$. Then there exist positive constants $c_0$ and $\delta_0 $ such that for every $0<\delta\leq \delta_0$ one has
\begin{equation}\label{MR1}
\limsup_{\varepsilon \downarrow 0} \varepsilon \log \mu^\varepsilon \Bigl\{ \text{\emph{dist}}_{L^2}(u,M) \geq \delta \Bigr\} \leq - c_0 \delta^2.
\end{equation}
In particular the measures $ \mu^\varepsilon$ concentrate around the set of minimizers exponentially fast.
\end{thm}
The crucial step in the proof is to find a lower bound on the exponential decay of the normalization constant $Z^\varepsilon$. This lower bound can be found in Section \ref{SEC4}.

\vskip3ex 

The same result also holds using the $L^\infty$-norm: 
\begin{thm}\label{THM3}
Assume $0<\gamma < \frac{2}{3}$. Then there exist positive constants $\tilde{c}_0$ and $\tilde{\delta}_0 $ such that for every $0<\delta \leq \tilde{\delta}_0$ one has
\begin{equation}\label{MR2}
\limsup_{\varepsilon \downarrow 0} \varepsilon \log \mu^\varepsilon \Bigl\{ \text{\emph{dist}}_{L^\infty}(u,M) \geq \delta \Bigr\} \leq - \tilde{c}_0 \delta^2.
\end{equation}
\end{thm}
\begin{figure}
\begin{center}
\includegraphics[height=6cm]{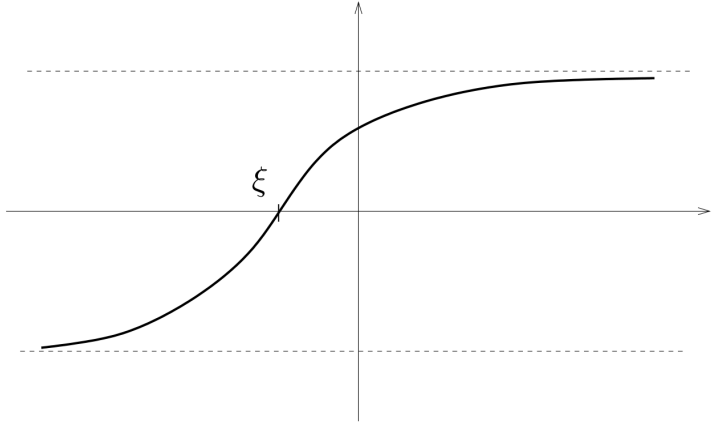}\label{StandingWave}
\caption{The instanton shape $m_\xi$.}
\end{center}
\end{figure}

\vskip3ex
 \newpage

\emph{3. Motivation and related works}

\vskip1ex

The Allen-Cahn equation without noise has been introduced in \cite{AC79} to model the dynamics of interfaces between different domains of different lattice structure in crystals and has been studied since in various contexts. In the one-dimensional case the dynamics of the deterministic equation are well-understood \cite{Ch04,CP89,OR07} and can be described as follows: If one starts with arbitrary initial data, solutions will quickly tend to configurations which are locally constant close to $\pm 1$ possibly with many transition layers that roughly look like the instanton shapes $m$ introduced above. Then these interfaces move extremely slowly until eventually some two transition layers meet and annihilate each other. After that the dynamics continue very slowly with less interfaces.  

\vskip3ex

In the higher dimensional case no such metastable behavior occurs. Also here solutions tend very quickly towards configurations which are locally constant with interfaces of width $\varepsilon$. Then on a slower scale these interfaces evolve according to motion by mean curvature (see \cite{Il93} and the references therein).

\vskip3ex

 Stochastic systems which are very similar to (\ref{SACE}) have been studied in the ninetees by Funaki \cite{Fu95} and Brasecco, de Masi, Presutti \cite{BMP95}. They study the one-dimensional equation in the case where the initial data is close to the instanton shape and show that in an appropriate scaling the solution will stay close to such a shape. Then due to the random perturbation a dynamic along the one-parameter family of such shapes can be observed on a much faster time scale than in the deterministic case. Our result Theorem \ref{THM1} says that one can also pass to the sharp interface limit on the level of invariant measures.

\vskip3ex

If the process does not start in a configuration with a single interface, it is believed that these different interfaces also follow a random induced dynamic which is much quicker than in the deterministic case. Different interfaces should annihilate when they meet \cite{FV03}. More recently there were also investigations of the same system on a much bigger space interval where due to entropic effects noise induced nucleation should occur. This phenomenon has been studied on the level of invariant measures \cite{VR05}. The limiting process should be related to the Brownian web which has recently been investigated e.g. in \cite{FINR06}.

\vskip3ex

From a point of view of statistical physics Theorem \ref{THM2} can be interpreted as quite natural. In fact the Feynman picture suggests to view $\mu^\varepsilon$ as a Gibbs measure with energy $\mathcal{H}$ and decreasing temperature $\varepsilon$. On a fixed interval the result of Theorem \ref{THM2} would therefore simply state that with decreasing temperature the Gibbs measure concentrates around the energy minimizers exponentially fast. On a rigorous level such results follow from standard Large Deviation Theory (see e.g. \cite{dH00,DS89}). Our result states that the entropic effects which originate from considering growing intervals do not change this picture. In fact also this is not very suprising - analysis of similar spin systems suggests that even on intervals that grow exponentially in $\varepsilon^{-1}$ one should not observe more than one jump. But it is not clear if one can say anything about the shape of the interface in this settings. Our approach is limited to intervals growing like $\varepsilon^{-\gamma}$ due to the $L^2$-Hilbert space structure employed.

\vskip3ex

\emph{4. Structure of the paper} 

\vskip1ex

In Section \ref{SEC2} results about the energy landscape of the Ginzburg-Landau energy functional are summarized. In particular we discuss in some detail the minimizers of $\mathcal{H}$ and introduce tubular coordinates close to the curve of minimizers. The energy landscape is studied in terms of these tubular coordinates. In Section \ref{SEC3} some necessary Gaussian concentration inequalities are discussed. In particular the discretization of the measure $\nu^\varepsilon$ is given and some error bounds are proved. The proof of Theorem \ref{THM2} can then be found in Section \ref{SEC4}. Finally the proof of Theorem \ref{THM1} is finished in Section \ref{SEC5}. We will follow the convention that $C$ denotes a generic constant which may change from line to line. Constants that appear several times will be numbered $c_1,c_2,\ldots$.

\section{The Energy Functional}\label{SEC2}
In this section we discuss properties of the Ginzburg-Landau energy functional. We introduce the one parameter family of minimizers which we think of as a one-dimensional submanifold of the infinite-dimensional space of possible configurations. Then we discuss tubular coordinates of a neighborhood of this curve as well as a Taylor expansion of the energy landscape in these tubular coordinates. These ideas are mostly classical and go back to \cite{CP89, Fu95, OR07}.  Finally we give a discretized version of the minimizers and proof some error bounds.

\vskip3ex

For a function $u$ defined on the whole real line consider the following energy functional:
\[
\mathcal{H}(u)=\int_\RR \left[ \frac{1}{2}  |u'(s)|^2 + F\big(u(s) \big) \right] \, \mathrm{ds} - C_* ,
\]
where the constant $C_*$ is chosen in a way to guarantee that the minimum of $\mathcal{H}$ on the set of functions with the right boundary conditions is $0$. In fact let $m$ be the standing wave solution of the Allen-Cahn equation:
\begin{equation}\label{SWS}
 m''(s)-F'(m(s))=0 \qquad \forall s \in \RR, \qquad m(\pm s) \rightarrow \pm 1 \quad \text{for} \quad s \rightarrow \infty.
\end{equation}
As (\ref{SWS}) is invariant under translations one can assume $m(0)=0$. Then the solution can be found by solving the system
\begin{equation}\label{SWS2}
 m'(s) - \sqrt{2F(m(s))}=0 \qquad \forall s \in \RR, \qquad  m(0)=0 \qquad m(\pm \infty)= \pm 1. 
\end{equation}
Note that the assumptions (\ref{assumptions}) on $F$ imply that $\sqrt{F}$ is $C^1$ such that the solution to (\ref{SWS2}) is unique. The translations of $m$ will be denoted by $m_\xi(s)=m(s-\xi)$. Note that the $m_\xi$ are not the only solutions to (\ref{SWS}) but that all the other solutions are either periodic or diverge such that the $m_\xi$ are the only nonconstant critical points of $\mathcal{H}$ with finite energy. In fact the $m_\xi$ are global minimizers of $\mathcal{H}$ subject to its boundary conditions. One has simply by completing the squares:
\begin{equation}\label{SWS9}
\begin{split}
\int_\RR \left[ \frac{1}{2}  |u'(s)|^2 + F \big(u(s) \big) \right] \, \mathrm{ds} & \, = \, \int_{-\infty}^{\infty}  \frac{1}{2} \Bigl(u'(s) - \sqrt{2F(u(s))}   \Bigr)^2 +  \sqrt{2F\big(u(s)\big)} u'(s) \,\mathrm{ds} \\
 &\, \geq \, \int_{u(-\infty)}^{u(\infty)} \sqrt{2F(u)} \, \mathrm{du}. 
\end{split}
\end{equation}
The term in the bracket is nonnegative and it vanishes if and only if $u$ solves (\ref{SWS2}). In the sequel we will write
\[
M=\{ m_\xi, \xi \in \RR \} \qquad \text{and} \qquad C_*= \int_\RR  \frac{1}{2} \left[ |m'(s)|^2 + F\big( m(s) \big) \right] \, \mathrm{ds}.
\]
For notational convenience we introduce the function $G(u)=\int_0^u \sqrt{2F(u)} \mathrm{du}$. Then equation (\ref{SWS9}) states that $\int_\RR  \frac{1}{2}  |u'(s)|^2 + F(u(s)) \, \mathrm{ds} \geq G \big(u(\infty) \big)-G\big( u(-\infty)\big)$. Note that the assumption (\ref{assumptions}) on $F$ imply that $G$ is a strictly increasing $C^4$ function with $G(0)=0$. In the case of the standard double-well potential $F(u)=\frac{1}{2} (u^2-1)^2$ a calculation yields
\[
m(s)= \tanh(s) \qquad \text{and} \qquad C_*=\frac{4}{3}.
\]
Equation (\ref{SWS2}) shows that in general $m$ can be given implicitly as
\begin{equation}
 s=\int_0^m \frac{1}{\sqrt{2F\big( \tilde{m} \big)}} \, \mathrm{d}\tilde{m}.
\end{equation}
By expanding $F$ around $1$ one obtains exponential convergence to $\pm 1$ for $s \rightarrow \pm 1$. To be more precise there exist positive constants $c_1$ and $c_2$ such that
\begin{equation}\label{tail}
 \begin{cases}
| 1 \mp m( \pm s)| \leq c_1  \exp(-c_2  s) \qquad & s \geq 0\\
|m'(\pm s)| \leq c_1 c_2  \exp(-c_2  s) \qquad & s \geq 0\\
|m''(\pm s)| \leq c_1 c_2^2  \exp(-c_2  s) \qquad & s \geq 0.
\end{cases}
\end{equation}
Recall that $ m+L^2(\RR)= \{ u \colon u-m \in L^2(\RR ) \}$. Note that for all $\xi$ due to (\ref{tail}) $m-m_\xi \in L^2(\RR)$ such that it does not matter which $\xi$ one takes for the definition of this space.

\vskip3ex

We now introduce the concept of Fermi coordinates which was first used in this context in \cite{CP89,Fu95}: Recall that for a function $u\in m+ L^2(\RR)$ we write $\text{dist}_{L^2}(u,M):= \inf_{\xi \in \RR} \|u- m_\xi  \|_{L^2(\RR)}$. If $\text{dist}_{L^2}(u,M)$ is small enough there exists a unique $\xi \in \RR$ such that  $\text{dist}(u,M)=  \|u- m_\xi  \|_{L^2(\RR)}$ and one has 
\begin{equation}\label{Normal}
\langle u-m_\xi, m_\xi' \rangle_{L^2(\RR)} =0.
\end{equation} 
In fact the last equality (\ref{Normal}) can easily be seen by differentiating $\xi \mapsto \| u- m_\xi \|^2_{L^2(\RR)}$. This has a simple geometric interpretation. The function $m_\xi'$ can be seen as tangent vector to the curve $M$ in $m_\xi$ and the relation (\ref{Normal}) can be interpreted as $v:=u-m_\xi$ being normal to the tangent space in $m_\xi$. We will denote the space 
\[
N_\xi := \{v \in L^2(\RR): \langle v, m_\xi' \rangle_{L^2(\RR)} =0   \}
\]
and interpret it as the normal space to $M$ in $m_\xi$. For $u= m_\xi + v$ with $v \in N_\xi $ we will call the pair $(\xi, v)$ Fermi or tubular coordinates of $u$. 

\vskip3ex

One obtains information about the behavior of the energy functional close to $M$ by considering the linearized Schr\"odinger type operators
\[
\mathcal{A}_\xi= - \Delta + F''(m_\xi).
\]
with domain of definition $H^2(\RR) \subset L^2(\RR)$. The operator $\mathcal{A}_\xi$ is selfadjoint and nonnegative (see e.g. \cite{Fu95}) and the unique the eigenspace corresponding to the eigenvalue $0$ is spanned by the function $m_\xi'$. This can be understood quite easily: The fact that the operator is nonnegative corresponds to the functional $\mathcal{H}$ attaining its minimum at $m_\xi$ and the fact that $m_\xi'$ is a eigenfunction to the eigenvalue $0$ corresponds to the translational invariance of $\mathcal{H}$. The following more detailed description of the spectral behavior of $\mathcal{A}_\xi$ is taken from \cite{OR07} Proposition 3.2 on page 391:
\begin{lemme}\label{Spec}
There exists a constant $c_3 > 0$ such that if $u \in H^1(\RR)$ satisfies 
\begin{equation*}
 (i) \quad  u(\xi) \,= \, 0 \qquad \text{ or } \qquad (ii) \quad \int_\RR u(s) \, m_\xi'(s)  \, \mathrm{ds} \,= \,0,
\end{equation*}
then 
\begin{equation}\label{SPEC}
c_3 \| u \|_{L^2(\RR)}^2 \, \leq \, \int_\RR \left[ u'(s)^2 + F''\big( m_\xi(s) \big) \, u(s)^2 \right] \, \mathrm{ds}. 
\end{equation}
\end{lemme}
This can be used to obtain the following description of the energy landscape. Similar results were already obtained in \cite{Fu95} and \cite{OR07}:
\begin{prop}\label{SPEC2}
\begin{enumerate}
 \item[(i)] There exist non-negative constants $c_0,c_4, \delta_1$ such that for $u$ with Fermi coordinates $u=m_\xi + v$ and $\|v \|_{H^1(\RR)} \leq \delta_1$ one has:
\begin{equation}\label{Spec5}
c_0 \|v \|_{H^1(\RR)}^2 \leq \mathcal{H}(u) \leq c_4 \|v \|_{H^1(\RR)}^2. 
\end{equation} 
\item[(ii)] There exists a $\delta_0 > 0$ such that for $\delta \leq \delta_0$ the relation $\text{\emph{dist}}_{H^1}(u,M) \geq \delta$ implies
\begin{equation}\label{GloBo}
\mathcal{H}(u) \geq c_0 \delta^2.
\end{equation} 
\end{enumerate}
\end{prop}
Here $\text{dist}_{H^1}(u,M)= \inf_{\xi \in \RR} \|u-m_\xi \|_{H^1(\RR)}$. Statement (i) will be used as a local description of the energy landscape close to the curve of minimizers whereas the statement (ii) will be useful as a rough lower bound for the energy away from the curve. For the proof of Proposition \ref{SPEC2} one needs the following lemma:
\begin{lemme}\label{inftybou}
For every $\varepsilon> 0$ there exists $\delta > 0 $ such that if $u \in m+ L^2$ with $\mathcal{H}(u) \leq \delta$ then there exists $\xi \in \RR $ such that 
\[
 \| u - m_\xi \|_{L^\infty(\RR)} \leq \varepsilon.
\]
Furthermore $\xi$ can be chosen in a such a way that $u(\xi)=0$. 
\end{lemme}
\begin{proof}
For a small $\delta>0$ assume $\mathcal{H}(u) \leq \delta$. We want to find a $\xi \in \RR$ such that by choosing $\delta$ sufficiently small we can deduce that $\|u - m_\xi \|_{L^\infty(\RR)}$ becomes arbitrarily small. As $\mathcal{H}(u) < \infty$ we have $u \in m + H^1$ and therefore in particular $u \in C^0(\RR) \cap L^\infty(\RR)$. Note that a similar calculation as (\ref{SWS9}) implies that $\mathcal{H}(u) \geq \left(  G\big(\sup_{s \in \RR}u(s) \big) -G\big(\inf_{s \in \RR}u(s) \big)  \right) - \left(  G(1) -G(-1) \right)$. Therefore by the properties of $G$ by choosing $\delta$ sufficiently small, one can assume that $\| u \|_{ L^\infty(\RR)} \leq 2$. By the assumptions (\ref{assumptions}) on $F$ there exists a $C$ such that for $u \in [-2,2]$ one has 
\[
 F(u) \geq C \min \left( |u-1|,|u+1|  \right)^2,
\]
and in particular we know that for every interval $I$ the $H^1$-norm of $\min \left(|u-1|,|u+1|  \right)$ can be controlled by the energy. As $u$ is continuous and converges to $\pm 1$ as $s$ goes to $\pm \infty$, there exist a $\xi$ with $u(\xi)=0$. Without loss of generality one can assume that $\xi =0$. We will show that in this case $\|u - m \|_{L^\infty(\RR)}$ can be made arbitrarily small.

\vskip3ex

According to (\ref{tail}) for every $\varepsilon > 0$ there exists $T$ such that for $s \geq T$ one has $|u(s) -1 | \leq \varepsilon$ and for $s \geq T$ it holds that  $|u(s) + 1 | \leq \varepsilon$. We will first give a bound on $u-m$ in $[-T,T]$. We consider only the case $s \geq 0$ the other one being similar. Note that as according to (\ref{SWS9})
\[
 \mathcal{H}(u) = \int_\RR \frac{1}{2} \left(u'(s) -\sqrt{2 F(u)} \right)^2  \mathrm{ds},
\]
one can write 
\begin{equation}\label{SWS10}
 \begin{split}
 u'(s)&\, = \, \sqrt{F\big(u(s) \big)} +r(s)\\
u(0)& \,= \,0 
 \end{split}
\end{equation}
where $\int_0^T r(s)^2 \mathrm{ds} \leq 2\delta$ and therefore using Cauchy-Schwarz inequality
\[
 \int_0^T |r(s)| \mathrm{ds} \, \leq \, \sqrt{2T \delta}.
\]
Thus using (\ref{SWS2}) one obtains for $v=u-m$
\begin{equation}\label{SWS11}
 \begin{split}
v'(s) & \, =  \,  \sqrt{F\big(u(s)\big)} - \sqrt{F\big(m(s)\big)}  +r(s) \, \leq \, C v(s) +r(s)\\
v(0)& \,= \, 0,
 \end{split}
\end{equation}
where the constant $C$ is given by $C= \sup_{u \in [-2,2]} \frac{\mathrm{d}}{\mathrm{du}} \left( \sqrt{F(u)} \right)$. Thus Gronwall's Lemma implies
\[
| v(s)| \, \leq \, \int_0^s r(t) e^{C(s-t)} \, \mathrm{dt},
\]
and therefore $\sup_{s \in [0,T]} |v(s)| \leq \sqrt{2T \delta} e^{CT}$. Thus by choosing $\delta$ small enough one can assure that $\sup_{s \in [0,T]} |v(s)| \leq \frac{\varepsilon}{2}$. 

\vskip3ex

Now let us focus on the case $s \in [-T,T]^c$. We will again only focus on $s \geq T$. Note that by the above calculations  and the choice of $T$ one has $u(-T) \leq 1 -  \varepsilon$ and $u(T) \geq 1 -\varepsilon$. Therefore using 
\[
 \int_{-\infty}^{-T} \frac{u'(s)^2}{2}+ F(u(s)) \mathrm{ds} +  \int_{-T}^{T} \frac{u'(s)^2}{2}+ F(u(s)) \mathrm{ds} +\int_{T}^{\infty} \frac{u'(s)^2}{2}+ F(u(s)) \mathrm{ds} \leq G(1)-G(-1) +\delta, 
\]
as well as 
\[
 \int_{-T}^{T} \frac{u'(s)^2}{2}+ F(u(s)) \mathrm{ds} \geq G(T)-G(-T),
\]
we get 
\[
\int_{T}^{\infty} \frac{u'(s)^2}{2}+ F(u(s)) \mathrm{ds} \leq \left( G(1)-G(T) \right)-\left( G(-1) -G(-T) \right) +\delta \leq C \varepsilon +\delta, 
\]
where $C= 2 \sup_{u \in [-2,2]} F(u)$.Thus by using the fact that  $\int_{T}^{\infty} \frac{u'(s)^2}{2}+ F(u(s))$ controls the $H^1$-norm and therefore the $L^\infty$ of $ \min \left(|u-1|,|u+1|  \right)$ on $[T,\infty)$, one can conclude, that possibly by choosing a smaller $\delta$ one obtains $\sup_{s \in [t,\infty)} v(s) \leq C \varepsilon$. Thus by redefining $\varepsilon$ one obtains the desired result.
\end{proof}

\begin{proof}(Of Proposition \ref{SPEC2}):
(i) First of all remark that  for $v \in N_\xi$ one has
\begin{equation}\label{Spec2}
 \tilde{c_0} \|v \|_{H^1(\RR)}^2 \leq \langle v, \mathcal{A_\xi} v \rangle_{L^2(\RR)} \leq \tilde{c_4} \|v \|_{H^1(\RR)}^2.
\end{equation}
In fact Lemma \ref{Spec} (ii)  implies that
\begin{equation}\label{Spec4}
c_3 \| v \|^2_{L^2(\RR)} \leq  \langle v, \mathcal {A_\xi} v  \rangle_{L^2(\RR)} .
\end{equation}
To get the lower bound in (\ref{Spec2}) write
\begin{equation} \label{Spec3}
\begin{split}
 \langle  \mathcal{A} v,v \rangle_{L^2(\RR)} = \| \nabla v \|^2_{L^2(\RR)} + \int_{\RR} F''(m(y)) v^2(s) \mathrm{ds}\\
\geq \|v \|^2_{H^1(\RR)} - (c_5+1) \| v \|_{L^2(\RR)}^2 ,
\end{split}
\end{equation}
where $c_5= \max_{|v| \leq 1} F''(v)$. Then (\ref{Spec2}) follows with $\tilde{c_0}= \frac{\mu_*}{\mu_* c_0 +1}$. In fact if $\| v \|_{L^2} \leq \|v \|_{H^1} \frac{1}{c_3+\tilde{c}_0}$ one can use (\ref{Spec3}) and one can use (\ref{Spec4}) else. The upper bound in (\ref{Spec2}) is immediate noting that $\sup_{u \in [-1,+1]}|F''(u)| < \infty$.

\vskip2ex

In order to obtain (\ref{Spec5}) one writes:
\begin{equation}\label{Spec6}
 \mathcal{H}(u)= \frac{1}{2} \langle \mathcal{A_\xi}v,v \rangle + \int_\RR U(s,\xi,v) \mathrm{ds},
\end{equation}
where 
\[
 U(s,\xi,v)= F(m_\xi(s)+v(s))+F(m_\xi(s))-F'(m_\xi(s))v(s)-\frac{1}{2}F''(m_\xi(s))v(s)^2. 
\]
Here equation (\ref{SWS}) is used. Using that by Sobolev embedding $\| v \|_{L^\infty(\RR)} \leq C \| v \|_{H^1(\RR)}$ one obtains by Taylor formula
\begin{equation}\label{Spec7}
 \Bigl| \int_\RR U  \Bigr| \leq \frac{1}{6} \sup_{|v| \leq C \delta_1+1} |F'''(v)| \| v \|_{L^3(\RR)}^3 \leq C \| v \|_{L^\infty(\RR)} \| v \|_{L^2(\RR)}^2 \leq C \| v \|_{H^1(\RR)}^3.
\end{equation}
This implies the inequality (\ref{Spec5}).

\vskip3ex

(ii) To show the second statement, first note that there exists a $\tilde{\delta}_0 > 0$ such that if $\mathcal{H}(u) \leq \tilde{\delta}_0$ there exists a $\xi$ such that 
\begin{equation}\label{GloBo1}
 c_0 \| u - m_\xi \|_{H^1(\RR)}^2  \leq \mathcal{H}(u). 
\end{equation}
In fact choosing $\xi$ as in Lemma \ref{inftybou} and noting that if one uses the case (i) of Lemma \ref{Spec} instead (ii) one sees that inequalities (\ref{Spec2}) and (\ref{Spec7}) remain valid for $v= u - m_\xi$. Then by using the $L^\infty$ bound on $v$ from Lemma \ref{inftybou} instead of Sobolev embedding in the last step of (\ref{Spec7}) one obtains the above statement. In order to obtain (\ref{GloBo}) choose $\delta_0= \frac{\tilde{\delta}_0}{c_0}$ and assume $\text{dist}_{H^1}(u,M) \geq \delta$ for a $\delta \leq \delta_0$. If $\mathcal{H}(u) \geq \tilde{\delta}_0$ the bound (\ref{GloBo}) holds automatically. Otherwhise (\ref{GloBo1}) holds and gives the desired estimate.   
\end{proof}

\vskip3ex

We now pass to some bounds on approximated wave shapes. To this end fix $\gamma_1 < \gamma$. This parameter will be fixed throughout the paper. Denote by $m^\varepsilon$ the profile $m$ cut off outside of $[-\varepsilon^{-\gamma_1},\varepsilon^{-\gamma_1}]$. More precisely assume that $m^\varepsilon$ is a smooth monotone function that coincides with $m$ on $[-\varepsilon^{-\gamma_1},\varepsilon^{-\gamma_1}]$ and that verifies $m^\varepsilon(s)= \pm 1$ for $\pm s \geq \varepsilon^{-\gamma_1}+1$. Assume furthermore that on the intervals $[\varepsilon^{-\gamma_1}, \varepsilon^{-\gamma_1}+1]$ (respectively $[-\varepsilon^{-\gamma_1}-1, -\varepsilon^{-\gamma_1}]$) one has $u(s) \leq u^\varepsilon(s) \leq 1$ (resp. $u(s) \geq u^\varepsilon(s) \geq -1$). Due to (\ref{tail}) one can also assume that $|(u^\varepsilon)' (s)| \leq 2 c_1 c_2 e^{-c_2 \varepsilon^{-\gamma_1}}$ on both of these intermediate intervals. Then define $m_\xi^\varepsilon(s)=m^\varepsilon(s -\xi)$. 

\vskip3ex

Furthermore for $N \in \NN$ and $k \in \{-N,-(N-1),\ldots, (N-1),N \}$ set $s_k^{N,\varepsilon}=  \frac{k \varepsilon^{-\gamma}}{N}$ and define
\begin{equation}\label{discinst}
 m^{N,\varepsilon}_\xi(s)=
\begin{cases}
 m^\varepsilon_\xi(s) \qquad &\text{if $s = s_k^{N,\varepsilon} $ for $k=-(N-1), \ldots, (N-1)$}\\
&\text{the linear interpolation between these points,}
\end{cases}
\end{equation}
\begin{figure}
\hspace{-1cm}\includegraphics[height=4.7cm]{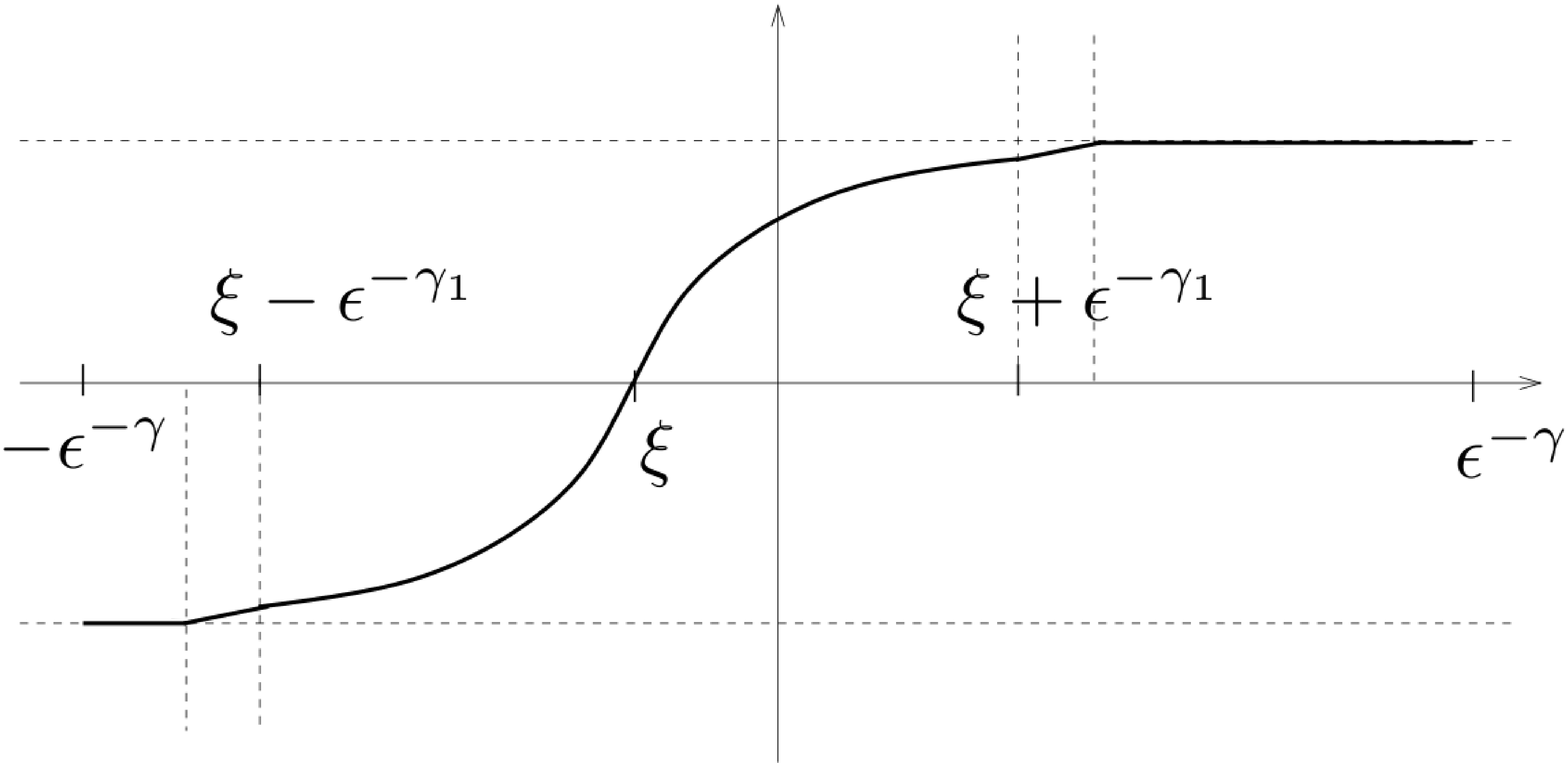}
\hspace{-1cm}\includegraphics[height=4.7cm]{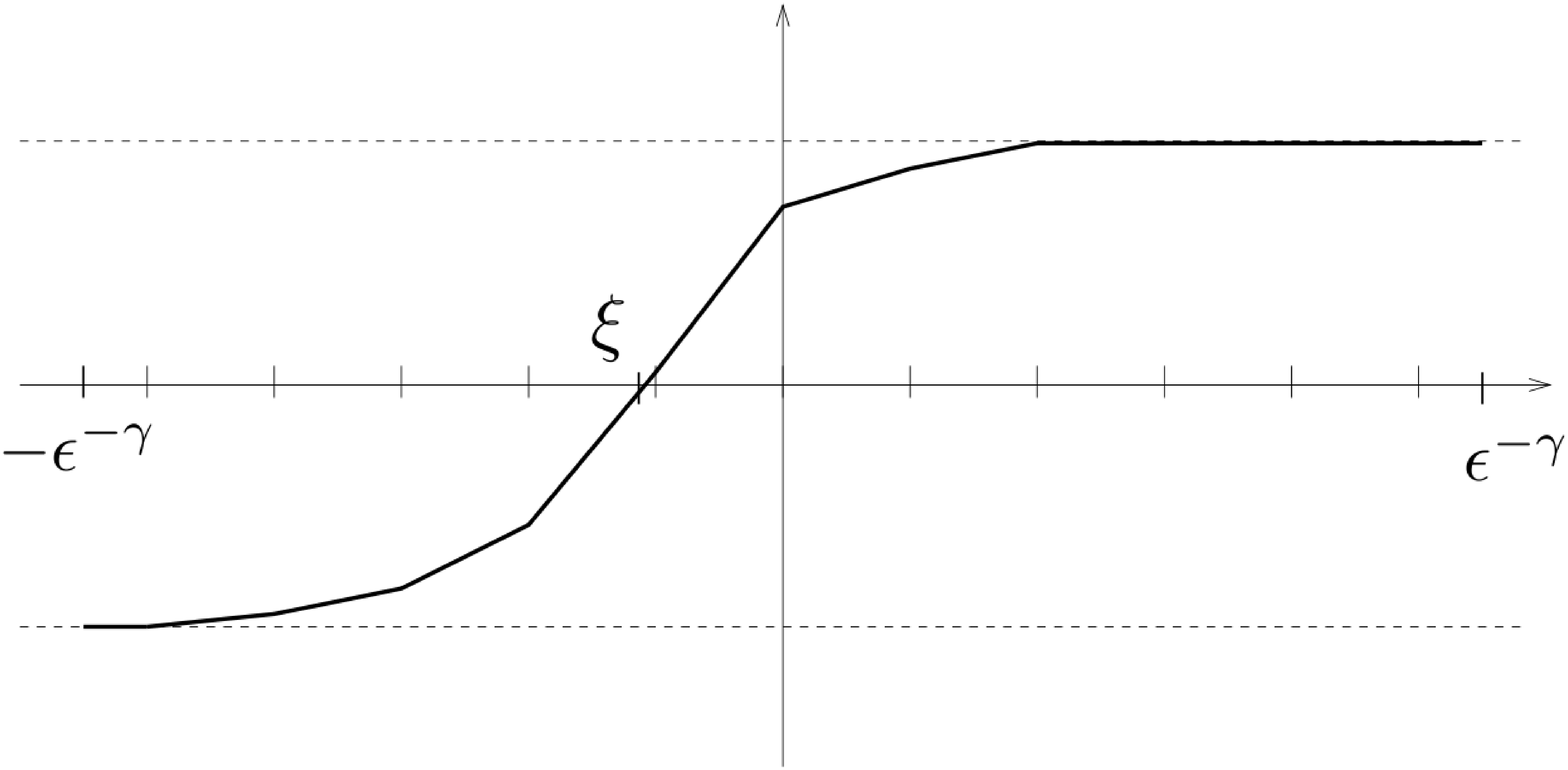}
\caption{The approximated waveshapes function $m_\xi^\varepsilon$ and   $m^{N,\varepsilon}_\xi$.}
\end{figure}

One then gets the following bound:
\begin{lemme}\label{discbo1}
For $\varepsilon$ small enough and $\xi \in [-\varepsilon^{-\gamma}+\varepsilon^{-\gamma_1}+1,\varepsilon^{-\gamma}-\varepsilon^{-\gamma_1}-1]$ one has
\begin{enumerate}
 \item[(i)]$\|m_\xi - m_\xi^\varepsilon  \|_{L^2(\RR)} \leq C \exp(-c_2\varepsilon^{-\gamma_1})$ and $\| (m_\xi)' -  (m_\xi^\varepsilon)'  \|_{L^2(\RR)} \leq C \exp(-c_2 \varepsilon^{-\gamma_1})$.
 \item[(ii)] $\|m_\xi-m_\xi^{N,\varepsilon} \|_{L^2(\RR)} \leq C \varepsilon^{-\gamma_1/2} \frac{\varepsilon^{-2\gamma}}{N^2}$ and $\| (m_\xi)'- (m_\xi^{N,\varepsilon})' \|_{L^2(\RR)} \leq C \varepsilon^{-\gamma_1/2} \frac{\varepsilon^{-\gamma}}{N}$.
\end{enumerate}
\end{lemme}
\begin{proof}
 To see (i) write
\begin{align*}
\|m_\xi - m_\xi^\varepsilon  \|_{L^2(\RR)}^2 &\leq  \int_{\varepsilon^{-\gamma_1}}^\infty \Bigl( m(s)-m^\varepsilon(s)\Bigr)^2 \mathrm{ds}  +\int_{-\infty}^{-\varepsilon^{-\gamma_1}} \Bigl( m(s)-m^\varepsilon(s)\Bigr)^2 \mathrm{ds}\\ 
&\leq 2  \int_{\varepsilon^{-\gamma_1}}^\infty c_1^2 \exp(- 2 c_2 s)\mathrm{ds} \leq C \exp(-2 c_2 \varepsilon^{-\gamma_1})
\end{align*}
and
\begin{align*}
\| m_\xi' -  (m_\xi^\varepsilon)'  \|_{L^2(\RR)}^2 &\leq  \int_{\varepsilon^{-\gamma_1}}^\infty \Bigl( m'(s)-(m^\varepsilon)'(s)\Bigr)^2 \mathrm{ds} + \int_{-\infty}^{\varepsilon^{-\gamma_1}} \Bigl( m'(s)-(m^\varepsilon)'(s)\Bigr)^2 \mathrm{ ds}\\  
&\leq C \exp(-2 c_2 \varepsilon^{-\gamma_1}).
\end{align*}
Here one uses the inequalities (\ref{tail}) as well as the properties of $m^\varepsilon$.
\vskip3ex
To see (ii) write 
\begin{align}\label{DB2}
\| m_\xi'- (m_\xi^{N,\varepsilon})' \|_{L^2(\RR)}\leq \| m_\xi'- (m_\xi^{\varepsilon})' \|_{L^2(\RR)}+\| (m_\xi^{\varepsilon})'- (m_\xi^{N,\varepsilon})' \|_{L^2(\RR)}.
\end{align}
To bound the second term assume without loss of generality that $\xi=0$ and write
\begin{align}\label{DB1}
\| (m^{\varepsilon})'- (m^{N,\varepsilon})' \|_{L^2(\RR)}^2 &= \sum_{k=-N}^{N-1} \int_{s_k^{N,\varepsilon}}^{s_{k+1}^{N,\varepsilon}}\Bigl( ( m^\varepsilon)'(s)- (m^{N,\varepsilon})'(s) \Bigr)^2 \mathrm{ds} \nonumber\\ 
&=\sum_{k=-N^\varepsilon}^{N^\varepsilon-1} \int_{s_k^{N,\varepsilon}}^{s_{k+1}^{N,\varepsilon}}\Bigl( ( m^\varepsilon)'(s)- (m^{N,\varepsilon})'(s) \Bigr)^2 \mathrm{ds}.
\end{align}
In the second equality $N^\varepsilon=  \lceil \varepsilon^{-\gamma_1} \frac{N}{\varepsilon^{-\gamma}} \rceil$. Here we use the fact that $u^\varepsilon$ is constant outside of $[-\varepsilon^{-\gamma_1},\varepsilon^{-\gamma_1}]$ and therefore coincides with its piecewise linearization. The integrals can be bounded using Poincar\'e inequality:
\begin{align}
\int_{s_k^{N,\varepsilon}}^{s_{k+1}^{N,\varepsilon}}\Bigl( ( m^\varepsilon)'(s)- (m_\xi^{N,\varepsilon})'(s) \Bigr)^2 \mathrm{ds} \leq \frac{\varepsilon^{-2\gamma}}{N^2 \pi^2} \int_{s_k^{N,\varepsilon}}^{s_{k+1}^{N,\varepsilon}}  (m^\varepsilon)''(s) \mathrm{ds}
\leq \frac{\varepsilon^{-3\gamma}}{N^3 \pi^2} \sup_{s \in \RR }|(m^\varepsilon)''(s)|^2.
\end{align}
Plugging this into (\ref{DB2}) one gets:
\begin{align*}
\| (m^{\varepsilon})'- (m^{N,\varepsilon})' \|_{L^2}^2 \leq  \varepsilon^{-\gamma_1} \frac{\varepsilon^{-2\gamma}}{N^2 \pi^2} \sup_{s \in \RR }|(m^\varepsilon)''(s)|^2.
\end{align*} 
Due to (i) the term involving $| m_\xi'- (m_\xi^{\varepsilon})' |$ can be absorbed in the constant for $\varepsilon$ small enough. This yields the second estimate in (ii). For the bound on $\| m_\xi'- (m_\xi^{\varepsilon})' \|_{L^2(\RR)}$ one proceeds in the same manner with another use of Poincar\'e inequality. The details are left to the reader. 
\end{proof}

\section{Gaussian estimates}\label{SEC3}
In this section concentration properties of some discretized Gaussian measure are discussed and the bounds which are needed in Section \ref{SEC4} are provided. To this end we recall a classical Gaussian concentration inequality. Then we introduce the discretized version of the Gaussian reference measure $\nu^\varepsilon$ and give an error bound. We also study another discretized measure which can be viewed as a discretized massive Gaussian free field.

\vskip3ex

Let $E$ be a separable Banach space equipped with its Borel-$\sigma$-field $\mathcal{F}$ and norm $\| \cdot \|$. Recall that a probability measure $\mu$ on $(E, \mathcal{F})$ is called Gaussian if for every $\eta$ in the dual space $X^*$ the pushforward measure $\eta_\# \mu$ is Gaussian.  For the moment all Gaussian measures are assumed to be centered i.e. for all $\eta \in X^*$ it holds  $\int \langle \eta, x\rangle  \mu(\mathrm{d}x)=0$. Denote by
\[
 \sigma= \sup_{\eta \in X^*, \|\eta \|_{X^*} \leq 1 } \left( \int \langle \eta, x \rangle^2 \mu(\mathrm{d} x)   \right)^{1/2}.
\]
Note that $\sigma$ is finite \cite{Le96}. Then one has the following classical concentration inequality (see \cite{Le96} page 203):
\begin{equation*}
\mu \Big(y; \, \| y \|  \, \geq  \, \int \|x \|  \,\mu(\mathrm{d}x)+r  \Big)  \, \leq \, e^{-r^2/2 \sigma^2}. 
\end{equation*}
In fact there are several ways to prove this, among them the Gaussian isoperimetric inequality.

\vskip3ex

The difficulty in applying this inequality to concrete examples is to evaluate the quantities $\sigma$ and $\int \|x \| \mu(\mathrm{d}x)$. This is easier in the case where $E$ is a Hilbert space. Then a centered Gaussian measure $\mu$ is uniquely characterized by the covariance operator $\Sigma$ which satisfies
\[
 \int \langle \eta_1, x \rangle \langle \eta_2, x \rangle  \, \mu(\mathrm{d}x) \,=  \, \langle \eta_1,\Sigma \eta_2 \rangle \qquad \forall \eta_1,\eta_2 \in E.
\]
It is known \cite{dPZ92} that $\Sigma$ must be a nonnegative symmetric trace class operator. Then $\sigma^2$ is the spectral radius of $\Sigma$ and using Jensen's inequality one obtains
\[
 \int \|x \| \, \mu(\mathrm{d}x) \leq  \bigl( \int \|x \|^2 \, \mu(\mathrm{d}x)\bigr)^{1/2}= \big( \text{Tr}\,\Sigma \big)^{1/2}.
\]
Therefore one can write
\begin{lemme}\label{GCII}
Let $\mu$ be a centered Gaussian measure on a Hilbert space $E$ with covariance operator $\Sigma$. Then one has
\begin{equation}\label{GCI2}
\mu \bigl(x; \| x \| \geq \big( \text{\emph{Tr}}\, \Sigma \big)^{1/2} +r  \bigr) \leq e^{-r^2/2 \sigma^2}. 
\end{equation}
\end{lemme}

\vskip3ex

We now want to use this inequality to study the behavior or the measure $\nu^\varepsilon$ under discretization. To this end fix an integer $N$ and consider piecewise affine functions $u\in L^2[-\varepsilon^{-\gamma},-\varepsilon^{-\gamma}]$ of the following type
\begin{equation}\label{DefH}
u(x)=
\begin{cases}
\pm 1 \qquad &\text{for} \quad x= \pm \varepsilon^{-\gamma}\\
\text{arbitrary} \qquad &\text{for} \quad x= s_k^{N,\varepsilon} \quad k=-(N-1),\ldots,(N-1)\\
&\text{the linear interpolation between those points},
\end{cases}
\end{equation}
and denote by $H^{N,\varepsilon}$ the affine space of all such functions. Recall that $s_k^{N,\varepsilon}=\frac{k \varepsilon^{-\gamma}}{N}$. The space $H^{N,\varepsilon}$ can canonically be identified with $\RR^{2N-1}$. In particular typical finite dimensional objects such as Lebesgue- and Hausdorff measures make sense on $H^{N,\varepsilon}$. On the other hand also the infinite dimensional observations from Section \ref{SEC2} can be applied to elements of $H^{N,\varepsilon}$. The interplay between infinite and finite dimensional ideas is crucial in the proof of Theorem \ref{THM2}.  Denote by $\mathcal{L}^{N,\varepsilon}$ the Lebesgue measure on $H^{N,\varepsilon}$. 

\vskip3ex

Recall that $\nu^\varepsilon$ is the distribution of a Gaussian process $\left(u(s), s \in [-\varepsilon^{-\gamma}, \varepsilon^{-\gamma}]\right)$ with $\EE[u(s)]=\varepsilon^\gamma s$ and $\text{Cov}(u(s),u(s'))= \varepsilon\Bigl(s \wedge s' +\varepsilon^{-\gamma}- \frac{(s +\varepsilon^{-\gamma})(s' +\varepsilon^{-\gamma})}{2 \varepsilon^{-\gamma}}  \Bigr)$. According to the Kolmogorov-Chentsov Theorem we can assume that $u$ has continuous paths. Consider now the piecewise linearization of $u^N$ of $u$: 
\[
u^N(s)=
\begin{cases}
\pm 1 \qquad &\text{for} \quad s= \pm \varepsilon^{-\gamma}\\
u(s) \qquad &\text{for} \quad x= s_k^{N,\varepsilon} \quad k=-(N-1),\ldots,(N-1)\\
&\text{the linear  interpolation between those points}.
\end{cases}
\]
\begin{lemme}\label{SiscBeh}
\begin{enumerate}
\item[(i)] The distribution of $u^N$  is absolutely continuous with respect to the Lebesgue measure $\mathcal{L}^N$ on $H^{N,\varepsilon}$. The density is given by
\begin{equation}\label{fdide}
\frac{1}{\sqrt{(2\pi)^{2N-1}}} \Bigl( \frac{N}{\varepsilon^{-\gamma}} \Bigr)^{N} \left( 2 \varepsilon^{-\gamma}\right)^{1/2} \exp\bigl( \varepsilon^{\gamma-1}) \exp \Bigl(- \frac{1}{\varepsilon} \int_{-\varepsilon^{-\gamma}}^{\varepsilon^{-\gamma}} |\nabla u(s)|^2 \mathrm{ds}\Bigr).
\end{equation}   
\item[(ii)] The random function $u-u^N$ consists of $2N$ independent rescaled Brownian bridges. To be more precise for each $k\in \{-N, \ldots (N-1) \}$ the process $(u(s)-u^N(s) \colon s \in [s_k^{N,\varepsilon},s_{k+1}^{N,\varepsilon}] )$ is a centered Gaussian process with covariance 
\begin{equation}\label{CovD}
\text{\emph{Cov}}(u(s)-u^N(s), u(s')-u^N(s'))= \varepsilon \Bigl( s \wedge s' -s_k^{N,\varepsilon} - \frac{(s - s_k^{N,\varepsilon} )(s' - s_k^{N,\varepsilon})}{\frac{\varepsilon^{-\gamma}}{N}} \Bigr). 
\end{equation}
These processes are mutually independent and independent of $u^N$. 
\end{enumerate}
\end{lemme}
\begin{proof}
(i) The measure $\nu^\varepsilon$ can be considered as the distribution of a rescaled Brownian $u$ motion on $[-\varepsilon^{-\gamma},\varepsilon^{-\gamma}]$ starting at $u(-\varepsilon^{-\gamma})=-1$ and conditioned on  $u(\varepsilon^{-\gamma})=1$. Therefore the finite dimensional distributions can be obtained by finite dimensional conditioning: 
\begin{align*}
 \nu^\varepsilon &\Bigl( u(s_{-(N-1)}^{N,\varepsilon})\in \mathrm{d}x_{-(N-1)}, \ldots, u(s_{N-1}^{N,\varepsilon})\in \mathrm{d}x_{(N-1)}  \Bigr) \\
&= \left( \prod_{i=-N}^{(N-1)} \frac{1}{\sqrt{(2 \pi) \delta}}\exp \Bigl(-\frac{(x_{i+1}-x_i)^2}{2 \varepsilon \delta}  \Bigr) \right) \left(  \frac{1}{\sqrt{(2 \pi)2 \varepsilon^{-\gamma}}}   \exp \Bigl(\frac{(1-(-1))^2}{4 \varepsilon^{-\gamma} \varepsilon } \Bigr)\right)^{-1} \\
&=     \frac{1}{\sqrt{(2 \pi)^{2N-1}}} \delta^{-N} 2 \varepsilon^{-\gamma} \exp(\varepsilon^{\gamma-1})   \exp \Bigl(-\frac{1}{2 \varepsilon} \sum_{i=-N}^{N-1} \delta     \frac{(x_{i+1}-x_i)^2}{ \delta^2}  \Bigr) 
\end{align*}
Here $ \delta=\frac{\varepsilon^{-\gamma}}{N}$ and $x_{\pm N}=\pm 1$. By noting that the Riemann sum appearing in the last line is equal to the integral of the squared derivative of the piecewise linearization one obtains the result.

\vskip3ex
 (ii) Denote for  $i= -N,\ldots , (N-1)$ and $s \in [0,\delta]$ by $\tilde{u}_i(s)= u(t_i+s)-u^N(t_i+s)= u(t_i+s)-\Bigl( 1-\frac{s}{\delta}\Bigr) u(t_i)- \frac{s}{\delta}u(t_{i+1})$. We want to show that the processes $(\tilde{u}_i(s), s \in [0,\delta])$ posses the right covariances and are mutually independent and independent of $u^N$. To this end calculate for $s,s'\in [0,\delta]$ and $i= -N,\ldots , (N-1)$:
\begin{align*}
 \text{Cov}&(\tilde{u}_i(s), \tilde{u}_i(s')) =\\
 & \qquad \text{Cov}\left[ u(t_i+s)-\Bigl( 1-\frac{s}{\delta}\Bigr) u(t_i)- \frac{s}{\delta}u(t_{i+1}), u(t_i+s')-\Bigl( 1-\frac{s'}{\delta}\Bigr) u(t_i)- \frac{s'}{\delta}u(t_{i+1})\right].
\end{align*}
By plugging in the explicit expression for the covariances of the $u(s)$ and some tedious but elementary calculations one obtains the desired expression. In a similar way one can see that for $i \neq j$ one has 
\[
 \text{Cov}(\tilde{u}_j(s), \tilde{u}_i(s'))=0 \qquad \text{and} \qquad  \text{Cov}(\tilde{u}_j(s), u^N(t))=0  
\]
for all $s,s' \in [0,\delta]$ and $t \in [-\varepsilon^{-\gamma},\varepsilon^{-\gamma}]$.
\end{proof}
Denote the Gaussian normalization constant
\[
Z_1^{N,\varepsilon}\colon=\frac{1}{\sqrt{(2\pi)^{2N-1}}} \Bigl( \frac{N}{\varepsilon^{-\gamma}} \Bigr)^{N} 2 \varepsilon^{-\gamma} \exp\bigl( \varepsilon^{\gamma-1}) 
\]
Note that by viewing $\nu^{N,\varepsilon}$ as finite dimensional measure with covariance given by the inverse of the negative Dirichlet Laplacian restricted to $H^{N,\varepsilon}$ which we denote by $- \Delta_N$ one sees that
\begin{equation}\label{Determ}
 Z_1^{N,\varepsilon}= \frac{1}{\sqrt{(2\pi)^{2N-1}}}\exp\bigl( \varepsilon^{\gamma-1}) \big( \text{det}(-\Delta_N) \big)^{1/2}.
\end{equation}

\vskip3ex
We now want to apply the Gaussian concentration inequality to obtain a bound on the probability of large $u-u^N$:
\begin{lemme}\label{DiscBound}
The following bounds hold:
\begin{enumerate}
 \item $L^2$-bound on the whole line:
\begin{equation}\label{Disbound1}
\nu^\varepsilon  \Bigl(u: \|u-u^N \|_{L^2[-\varepsilon^{-\gamma},\varepsilon^{-\gamma}]} \geq \sqrt{\varepsilon \frac{\varepsilon^{-2\gamma}}{3 N}}  +r  \Bigr) \leq \exp \left(-\frac{r^2 \pi^2 N^2}{\varepsilon^{1-2\gamma}} \right)
\end{equation}
\item $L^2$-bound on the short intervals:
\begin{equation}\label{Disbound2}
\nu^\varepsilon \left( \| u(s)-u^N(s)\|_{L^2[s_k^{N,\varepsilon},s_{k+1}^{N,\varepsilon}]}      \geq \sqrt{\varepsilon \frac{\varepsilon^{-2\gamma}}{6 N^2}} + r \right) \leq \exp \left(-\frac{r^2 \pi^2 N^2}{\varepsilon^{1-2\gamma}} \right).
\end{equation}
\item $L^\infty$-bound on the whole line:
\begin{equation}\label{Disbound3}
\nu^\varepsilon \left(  \| u(s) -u^N(s)\|_{L^\infty[-\varepsilon^{-\gamma},\varepsilon^{-\gamma}]}      \geq  r \right)  \leq  4N \exp \left(-\frac{ r^2 N  }{8 \varepsilon^{1-\gamma}  } \right) .
\end{equation}
\end{enumerate}
\end{lemme}
\begin{proof}
Let us consider (\ref{Disbound1}) first. Note that $u-u^N$ is a centered Gaussian process such that Lemma \ref{GCII} can be applied. The expected $L^2$-norm can be calculated as follows:
\begin{align*}
\nu^{\varepsilon} \Bigl[  \|u-u^N \|_{L^2[-\varepsilon^{-\gamma},\varepsilon^{-\gamma}]}^2    \Bigr]= \sum_{k=-N}^{N-1} \nu^{\varepsilon} \|\tilde{u}_k \|_{L^2}^2=\sum_{k=-N}^{N-1}  \int_{s_k^{N,\varepsilon}}^{s_{k+1}^{N,\varepsilon}} \nu^\varepsilon \bigl(\tilde{u}(s)^2 \bigr)\mathrm{ds}\\
= \sum_{k=-N}^{N-1}  \int_{s_k^{N,\varepsilon}}^{s_{k+1}^{N,\varepsilon}} \varepsilon \left(s-s_k^{N,\varepsilon} -\frac{\left(s-s_k^{N,\varepsilon} \right)^2}{\frac{\varepsilon^{-\gamma}}{N}}  \right)\mathrm{ds} =2N \varepsilon \frac{1}{6} \left(\frac{\varepsilon^{-\gamma}}{N} \right)^2.
\end{align*}
Here for the third equality equation (\ref{CovD}) is used.

\vskip3ex
To get information about the spectral radius of the covariance operator $\Sigma$ calculate for $f,g \in L^2[-\varepsilon^{-\gamma},\varepsilon^{-\gamma}]$:
\begin{align*}
 \langle \, f,\Sigma   g  \, \rangle & \, = \, \nu^\varepsilon \left[ \langle \,  f,u-u^N \, \rangle \langle \, g,u-u^N \, \rangle  \right]\\
&\, = \, \sum_{k=-N}^{N-1} \int_{s_k^{N,\varepsilon}}^{s_{k+1}^{N,\varepsilon}} \varepsilon \bigg( s \wedge s' -\frac{(s - s_k^{N,\varepsilon} )(s' - s_k^{N,\varepsilon})}{\frac{\varepsilon^{-\gamma}}{N}} \bigg)f(s)g(s') \, \mathrm{ds}.
\end{align*}
Here in the last step the independence of the different bridges is used as well as formula (\ref{fdide}). Note that the integral kernel in the last line is the Green function of the negative Dirichlet-Laplace operator on the interval $[s_k^{N,\varepsilon},s_{k+1}^{N,\varepsilon}]$. Denoting this operator by $\varepsilon ( - \Delta_{T_k})^{-1}$ one can therefore write 
\[
\langle f,\Sigma g \rangle \, = \, \sum_{k=-N}^{N-1} \langle f, \varepsilon ( - \Delta_{T_k})^{-1} g \rangle_{L^2(T_k)}.
\]
The spectral decomposition of the inverse Dirichlet-Laplace operator on intervals of length $T$ is well known. In fact on $L^2[0,T]$ the smallest  eigenvalue $\lambda_0$ and the according eigenfunction $e_0(x)$ are given as:
\[
 e_0(s)= \sin \Bigl( \frac{\pi s}{T}  \Bigr) \quad \text{and} \quad \lambda_0= \frac{\varepsilon T^2}{\pi^2 }.
\]
The spectral radius of $\varepsilon \big( - \Delta_{T_k} \big)^{-1}$ is thus given as 
\[
 \sigma^2_k = \varepsilon \frac{\varepsilon^{-2\gamma}}{(N\pi)^2}  .
\]
Therefore one can write
\begin{align*}
\sigma^2&\,= \, \sup_{f, \|f\|=1} \langle f,\Sigma f \rangle = \sup_{f, \|f\|=1} \sum_{k=-N}^{N-1} \langle f, \varepsilon (\Delta_k)^{-1} g \rangle_{L^2(T_k)}\\
&\, \leq \, \sup_{f, \|f\|=1} \sum_{k=-N}^{N-1} \sigma^2_k \langle f,f \rangle_{L^2(T_k)} = \varepsilon \left( \frac{\varepsilon^{-\gamma}}{\pi N}\right)^2 \sup_{f, \|f\|=1} \langle f,f \rangle. 
\end{align*}
 On the other hand by taking $f$ as a linear combination of the eigenfunctions on the shorter intervals one obtains 
\[
\sigma^2 =\varepsilon \left( \frac{ \varepsilon^{-\gamma}}{\pi N} \right)^2.
\]
Thus equation (\ref{GCI2}) gives the desired result. The proof of (\ref{Disbound2}) proceeds in the same manner. 

\vskip3ex
To prove the  third statement (\ref{Disbound3}) note that by Lemma \ref{SiscBeh}, the deviations of a the random function $u$ from the piecewise linearizations $u^N$ between the points $s_k^{N,\varepsilon}$ are independent Brownian bridges. Therefore such a process $\left( u(s_k^{N,\varepsilon}+s)-u^N (s_k^{N,\varepsilon}+s), 0 \leq s \leq \frac{\varepsilon^{-\gamma}}{N} \right)$ has the same distribution as $\varepsilon^{\frac{1}{2}} \left(B_s  - \frac{sN}{\varepsilon^{-\gamma}}  B_{\frac{\varepsilon^{-\gamma}}{N}} \right)$ for a Brownian motion $B$ defined on a probability space $(\Omega, \mathcal{F}, \PP)$. Therefore one can write
\begin{align*}
\nu^\varepsilon \left(  \| u(s) -u^N(s)\|_{L^\infty[-\varepsilon^{-\gamma},\varepsilon^{-\gamma}]}      \geq  r \right) &\leq \sum_{k=-N}^{N-1} \nu^\varepsilon \left( \max_{s_k^{N,\varepsilon} \leq s \leq s_{k+1}^{N,\varepsilon}} | u(s)-u^N(s)| \geq r   \right)  \\ 
&\leq 2N  \quad \PP \left(\max_{0 \leq s \leq \frac{\varepsilon^{-\gamma}}{N}} \left| \varepsilon^{1/2} \left( B_s  - \frac{sN}{\varepsilon^{-\gamma}}  B_{\frac{\varepsilon^{-\gamma}}{N}}   \right) \right|     \geq  r     \right)\\
&\leq 2N  \quad \PP \left(\max_{0 \leq s \leq \frac{\varepsilon^{-\gamma}}{N}} \left|  B_s     \right|     \geq  \frac{r}{2\varepsilon^{1/2}}     \right).
\end{align*}
Using the exponential version of the maximal inequality for martingales (see Proposition 1.8 in Chapter II in \cite{RY99}) one can see that
\[
\nu^\varepsilon \left(  \| u(s) -u^N(s)\|_{L^\infty[-\varepsilon^{-\gamma},\varepsilon^{-\gamma}]}      \geq  r \right)  \leq 4N \exp \left(-\frac{ r^2 N  }{8 \varepsilon^{1-\gamma}  } \right) . 
\]
\end{proof}

\vskip 3ex

We will denote the distribution on of $u^N$ as $\nu^{N, \varepsilon}$. Note that the last statement can also be interpreted as a statement on a coupling of $\nu^\varepsilon$ and $\nu^{N,\varepsilon}$. In fact let $\lambda^{N, \varepsilon}$ be the joint distribution of $u$ and its discretization $u^N$. Then Lemma \ref{DiscBound} states that
\begin{equation}\label{fddeb}
 \lambda^{N,\varepsilon} \left\{(u,u') \colon \|u -u' \|_{L^2(\RR)} \geq \sqrt{\varepsilon \frac{\varepsilon^{-2\gamma}}{3 N}}  +r  \right\} \leq \exp \left(-\frac{r^2 \pi^2 N^2}{\varepsilon^{1-2\gamma}} \right),
\end{equation}
and an analogous result for the $L^\infty$ norm. 

\vskip3ex

We now want to study the properties of another discrete Gaussian measure. In fact denote by $H^{N,\varepsilon}_0$ the space of affine functions defined as in \ref{DefH} with the only change that they are assumed to possess zero boundary conditions. The Lebesgue measure on this space is defined in the same manner. For a fixed constant $\kappa$ consider the  centered probability  measure $\varrho^{N,\varepsilon}$ whose density with respect to $\mathcal{L}^{N,\varepsilon}$ is proportional to 
\[
 \exp \left(-\kappa \frac{\int_{-\varepsilon^{-\gamma}}^{\varepsilon^{-\gamma}} |u(s)|^2 +|\nabla u(s)|^2  \quad \mathrm{ds}  }{2\varepsilon} \right).
\]
In fact this measure is a variant of what is known in the literature as discrete massive free field, discrete Ornstein-Uhlenbeck bridge or pinned $\nabla \phi$ surface model \cite{S07, HSVW05}. Denote the normalization constant
\[
 Z^{N,\varepsilon}_2= \int \exp \left(-\kappa \frac{\int_{-\varepsilon^{-\gamma}}^{\varepsilon^{-\gamma}} |u(s)|^2 +|\nabla u(s)|^2 \quad  \mathrm{ds} }{2\varepsilon} \right)\mathcal{L}^{N,\varepsilon}(\mathrm{d} u).
\]
\begin{lemme}\label{HGM}
\begin{enumerate}
 \item[(i)] $Z^{N,\varepsilon}_2$ is given as
\begin{equation}\label{ZN2}
\frac{1}{\sqrt{(2\varepsilon \kappa \pi)^{2N-1}}} \det( - \Delta_N+\text{\emph{Id}}).
\end{equation}
Recall that the operator $\Delta_N$ denotes the discretized Laplace operator introduced above equation (\ref{Determ}).
\item[(ii)] We have the following bound: For $r \geq 0$
\begin{equation}\label{GCI4}
 \varrho\left\{u\colon \|u \|_{H^1} \geq \frac{( 2N-1)  \varepsilon}{\kappa}+r      \right\} \leq \exp \left(-\kappa r^2/2 \varepsilon \right).
\end{equation}
\end{enumerate}
\end{lemme}
\begin{proof}
(i) To see this one only has to note that $\Delta_N+\text{Id}$ is the inverse covariance matrix of this finite dimensional Gaussian measure.
\vskip3ex
(ii) To see (\ref{GCI4}) write with a finite dimensional change of variables:
\begin{align*}  
\varrho \left\{ u\colon \|u \|_{H^1} \geq  r    \right\}=  \frac{1}{Z^{N,\varepsilon}_2} \int_{\{ u\colon \|u \|_{H^1} \geq  r \}}\exp\left(\kappa \frac{\|u\|_{H^1}^2 }{2\varepsilon} \right)\mathcal{L}^{N,\varepsilon}(\mathrm{d} u)\\
= \frac{1}{\sqrt{ \left(2\varepsilon \pi\right)^{2N-1}}} \int_{\{ \sum_{k=-N}^{N-1} x_k^2 \geq  r \}}\exp\left(\kappa \frac{\sum_{k=-N}^{N-1} x_k^2 }{2\varepsilon} \right)\mathrm{d}x_{-N} \ldots \mathrm{d}x_{N-1}.
\end{align*}
In fact here one uses the standart linear transformation that transforms a gaussian random variable on a finite dimensional space to a gaussian random variable with $\text{Id}$ covariance matrix. We have thus have to consider a vector of $2N-1$ independent centered Gaussian random variables $X_k$ with variance $\frac{\varepsilon}{\kappa}$. The expectation
\[
 \EE\left[  \sum_{k=-N}^{N-1} X_i^2 \right] =\frac{ 2N  \varepsilon}{\kappa}
\]
and the spectral radius
\[
 \sigma^2=\frac{\varepsilon}{\kappa}
\]
are calculated easily such that (\ref{GCI2}) gives the desired result.
\end{proof}

\section{Concentration around a curve in infinite dimensional space}\label{SEC4}
In this section we give the proof of Theorem \ref{THM2}. To this end we consider the finite dimensional measure
\[
 \mu^{N,\varepsilon}(\mathrm{d}u)=\frac{1}{Z^{N,\varepsilon}} \exp \left(-\frac{1}{\varepsilon} \int_{-\varepsilon^{-\gamma}}^{\varepsilon^{-\gamma}} F(u(s))ds  \right)\nu^{N,\varepsilon}( \mathrm{d}u),
\]
with the normalization constant $Z^{N,\varepsilon}= \int \exp \left(-\frac{1}{\varepsilon} \int F(u(s))\mathrm{ds}  \right)\nu^{N,\varepsilon}(\mathrm{d}u)$. Note that although $\nu^{N,\varepsilon}$ is given by the finite-dimensional marginals of $\nu^{\varepsilon}$, the measure $\mu^{N,\varepsilon}$ does not coincide with the finite dimensional distribution of $\mu^\varepsilon$. The strategy is now as follows: In Proposition \ref{LoBo} a lower bound on the discrete normalization constant $Z^{N,\varepsilon}$ is given. This is achieved by calculating the integral in a tubular neighborhood of the set of minimizers $M$. Then in Proposition \ref{UpBo} the rough energy bound given in Proposition \ref{GloBo} is used to conclude concentration of the discretized measure $ \mu^{N,\varepsilon}$ around the curve of minimizers. Finally Lemma \ref{normconst} gives a bound on the quotient $\frac{Z^{\varepsilon}}{Z^{N,\varepsilon}}$ which allows to finish the proof of concentration around the curve of minimizers also in the continuous case with the help of a coupling argument.

\vskip3ex

Recall the following version of the coarea formula:
\begin{lemme}\label{coar}
Let $f$ be a Lipschitz function $f:A \subseteq E \to I \subseteq \RR$, where $E$ is a $n$-dimensional Euclidean space and $A$ is an open subset and $I$ some interval. Denote by $\lambda^n, \lambda^1$ and $\mathcal{H}^{n-1}$ the Lebesgue measure on $E$, on $\RR$ and the $(n-1)$-dimensional Hausdorff measure on $E$ respectively. Suppose that the gradient (which exists $\lambda^n$-a.e.) $Df$ does not vanish $\lambda^n$ a.e. in $A$. Then for every nonnegative measurable test function $\varphi:A \rightarrow \RR$ one has the following formula:
\begin{equation}\label{Coarea}
\int_A \varphi(x)\lambda(\mathrm{d}x) = \int_I \lambda^1 (\mathrm{d} \xi) \int_{f^{-1}(\xi)} \mathcal{H}^{n-1}(\mathrm{d}x) \frac{1}{|Df(x)|_E} \varphi(x).
\end{equation}
\end{lemme}
In order to apply this formula \ref{Coarea} to $\mu^{N,\varepsilon}$ one needs the following:
\begin{lemme}
Consider the function $f:A \to I$, where $A_1:= \{ x \in m+L^2 \colon \text{\emph{dist}}_{L^2}(x,M) < \beta \}$ is the open set in which the Fermi coordinates are defined and $I = [-\varepsilon^{-\gamma}+\varepsilon^{-\gamma_1}, \varepsilon^{-\gamma}-\varepsilon^{-\gamma_1}]$, defined by
\[
 f(x)=f(m_\xi+s)=\xi,
\]
where $x=m_\xi+s$ are the Fermi coordinates of $x$. Then $f$ is Fr\'echet differentiable and one has
\begin{equation}\label{Frech}
 Df(x)[h]=Df(m_\xi+s)[h]= \frac{-\langle m_\xi', h\rangle}{|m_\xi'|^2 - \langle s, m_\xi'' \rangle }.
\end{equation}
\end{lemme}
\begin{proof}
The differentiability follows from the implicit function theorem. To calculate the derivative at $x=m_\xi+s$ in direction $h$ consider the function 
\[
\Phi(v,w)= \langle m_\xi -m_{w} +s + vh, m_{w}' \rangle, 
\]
defined in an environment of $(0,\xi) \in \RR^2$. Noting that one has $\Phi(v,f(m_\xi +s+vh))=0$ one can write
\[
 0=\partial_v \Phi(v,f(m_\xi +s+vh))+ \partial_w \Phi(v,f(m_\xi +s+vh)) Df(m_\xi+s)[h].
\]
Observing that
\[
 \partial_v \Phi(v,f(m_\xi +s+vh))= \langle h, m_{\xi}' \rangle
\]
and
\[
 \partial_w \Phi(v,f(m_\xi +s+vh))= -\langle m_{\xi}',m_{\xi}' \rangle  + \langle s , m_{\xi}''  \rangle
\]
concludes the proof.
\end{proof}
We want to apply the coarea formula to the function $f$ just defined, restricted to $H^{N,\varepsilon}$. There is a slight inconvenience which originates from the fact that the norm of the gradient which appears in \ref{Coarea} is the norm in the finite dimensional space $E$ whereas the gradient of the function $f$ is a function in $L^2(\RR)$. To resolve this is the content of the next lemma:
\begin{lemme}\label{fdid}
Let $g:m+L^2(\RR) \rightarrow \RR$ be a Fr\'echet differentiable function and denote by $\nabla g(x)$ its $L^2$-gradient at point $x$. Consider then the function $\tilde{g}$ defined on $\RR^{2N-1}$ obtained by composition of the embedding $\RR^{2N-1} \rightarrow H^{N,\varepsilon}$ and $g$. Denote by $\tilde{\nabla} \tilde{g}$ its gradient. Then one has the following inequality:
\[
\| \tilde{\nabla} \tilde{g} \|_{\RR^{2N-1}} \leq 2\sqrt{\frac{\varepsilon^{-\gamma}}{N}} \| \nabla g \|_{L^2}.
\]
\end{lemme}
\begin{proof}
We calculate the derivative of $\tilde{g}$ in direction $\tilde{e}_k=(0, \ldots ,0, 1,0, \ldots 0)$ with the $1$ on $k$-th position. Embedding $\tilde{e}_k$ into $H^{N,\varepsilon}$ gives the hat-function
\begin{equation}\label{Hat}
e_k(s)= 
\begin{cases}
0 \qquad &\text{for } s \notin [s_{k-1}^{N,\varepsilon}, s_{k+1}^{N,\varepsilon}] \\
\frac{s-s_{k-1}^{N,\varepsilon}}{\frac{\varepsilon^{-\gamma}}{N}} \qquad &\text{for } s \in ]s_{k-1}^{N,\varepsilon}, s_{k}^{N,\varepsilon}]\\
\frac{s_{k+1}^{N,\varepsilon}-s}{\frac{\varepsilon^{-\gamma}}{N}} \qquad &\text{for } s \in ]s_{k-1}^{N,\varepsilon}, s_{k}^{N,\varepsilon}].
\end{cases}
\end{equation}
Therefore one obtains
\[
 (\tilde{\nabla} \tilde{g})_k= \int_\RR e_k(s)  \nabla g (s)\mathrm{ds}= \int_{s_{k-1}^{N,\varepsilon}}^{s_{k+1}^{N,\varepsilon}}   e_k(s)  \nabla g (s)\mathrm{ds}.
\]
Applying Cauchy-Schwarz inequality and using $\|e_k\|_\infty \leq 1$ one gets:
\begin{equation}
\begin{split}
\| \tilde{\nabla} \tilde{g}\|_{\RR^{2N-1}}^2 &= \sum_{k=-(N-1)}^{N-1} \left( \int_{s_{k-1}^{N,\varepsilon}}^{s_{k+1}^{N,\varepsilon}} e_k(s)  \nabla g (s) \mathrm{ds} \right)^2 \leq 2\frac{\varepsilon^{-\gamma}}{N} \sum_{k=-(N-1)}^{N-1}  \int_{s_{k-1}^{N,\varepsilon}}^{s_{k+1}^{N,\varepsilon}} \left( \nabla g(s) \right)^2 \mathrm{ds} \\
 &\leq 2\frac{\varepsilon^{-\gamma}}{N} 2 \| \nabla g \|_{L^2{(\RR})}^2.
\end{split}
\end{equation}
\end{proof}

Now we are ready to derive a lower bound on the normalization constant $Z^{N, \varepsilon}$ of the finite dimensional approximation of $\mu^{\varepsilon}$. Recall that
$\mu^{N,\varepsilon}(\mathrm{d}u)= \frac{1}{Z^{N,\varepsilon}}\exp\left(-\frac{1}{\varepsilon}\int_{-\varepsilon^{-\gamma}}^{\varepsilon^{-\gamma}}F(u(s))\mathrm{ds} \right)\nu^{N,\varepsilon}(\mathrm{d}u)$ where $\nu^{N,\varepsilon}$ is a discretized Brownian bridge. One gets the following bound:
\begin{prop}\label{LoBo1}
If one chooses $N= N(\varepsilon)$ in a way that $\frac{\varepsilon^{-\gamma}}{N} \downarrow 0$ as $\varepsilon \downarrow 0$, then the following bound holds for $\varepsilon$ small enough and a small but fixed $\delta$: 
\begin{equation}\label{LoBo}
Z^{N,\varepsilon}  \geq  \exp \Bigl(-\frac{C_*}{\varepsilon}\Bigr)    \exp \left( C\frac{1}{\delta}\varepsilon^{-\gamma} \right) \exp\left(-C \varepsilon^{-\gamma_1}\left(\frac{\varepsilon^{-\gamma}}{N}   \right)^{1/2} \right)  \frac{\varepsilon^{-\gamma}}{N}   \exp \left( -2 C \left( \frac{\varepsilon^{-\gamma-\gamma_1/2}}{\varepsilon N} \right) \right)c_4^{-N} .
\end{equation}
In particular if one chooses $N=N(\varepsilon)$ growing like $\varepsilon^{-\gamma_2}$ and $\gamma_1$ small enough such that
\begin{align}
-\gamma_1 - \gamma/2 +\gamma_2 >0  \label{gamma1}\\
-\gamma -\gamma_1/2 +\gamma_2 > 0 \label{gamma2}\\
\gamma_2 < 1 \label{gamma3},
\end{align}
one obtains
\begin{equation}\label{LoBo10}
 \liminf_{\varepsilon \downarrow 0} \varepsilon \log Z^{N,\varepsilon} \geq - C_*.
\end{equation}
\end{prop}
\begin{proof}
Using the definition of $\nu^{N,\varepsilon}$ one can write
\begin{equation}\label{LoBo8}
\begin{split}
Z^{N,\varepsilon} & = \int_{H^{N,\varepsilon}}  \exp \Bigl(-\frac{1}{\varepsilon}\int_{-\varepsilon^{-\gamma}}^{\varepsilon^{-\gamma}}F(u(s))\mathrm{d}s  \Bigr)     \nu^{N, \varepsilon}(\mathrm{d}u)\\ 
&=\frac{1}{Z^{N,\varepsilon}_1}\exp\Bigl(-\frac{C_*}{\varepsilon}\Bigr)  \int_{H^{N,\varepsilon}}  \exp \Bigl(-\frac{1}{\varepsilon}\int_{-\varepsilon^{-\gamma}}^{\varepsilon^{-\gamma}}F(u(s))\mathrm{ds}     -\frac{1}{\varepsilon}\int_{-\varepsilon^{-\gamma}}^{\varepsilon^{-\gamma}}\frac{1}{2} |u'(s)|^2 \mathrm{ds} +\frac{C_*}{\varepsilon} \Bigr)   \mathcal{L}^{N,\varepsilon}(\mathrm{d} u)\\
&=\frac{1}{Z^{N,\varepsilon}_1}\exp\Bigl(-\frac{C_*}{\varepsilon}\Bigr)  \int_{H^{N,\varepsilon}}  \exp \Bigl(-\frac{1}{\varepsilon} \mathcal{H}(u)\Bigl) \mathcal{L}^{N,\varepsilon}(\mathrm{d}u).
\end{split}
\end{equation}
Recall that $Z^{N,\varepsilon}_1= \int \exp \Bigl(-\frac{1}{\varepsilon} \int_{-\varepsilon^{-\gamma}}^{\varepsilon^{-\gamma}}\frac{1}{2} |u'(s)|^2 \mathrm{ds} \Bigr)  \lambda^{N,\varepsilon}(\mathrm{d}u)$ is the normalization constant of the discretized Brownian bridge  and $\mathcal{L}^{N,\varepsilon}$ is the Lebesgue measure on the finite dimensional space $H^{N,\varepsilon}$.  In order to find a lower bound on $Z^{N,\varepsilon}$ we can restrict the integration to a tubular neighborhood of $M$. More precisely set $I_\varepsilon:=[-\varepsilon^{-\gamma}+\varepsilon^{-\gamma_1}, \varepsilon^{-\gamma}-\varepsilon^{-\gamma_1}]$ and
\[
 A_2:= \left\{u \in H^{N,\varepsilon}\colon u=m_\xi +v \colon   \langle v, m_\xi' \rangle_{L^2(\RR)} = 0   \text{ for some $\xi \in I_\varepsilon$ and } \| v \|_{H^1(\RR)} \leq \delta \right\},
\]
for some $\delta$ to be determined later. For the moment we will only assume $\delta$ to be small enough in order to be able to apply Funaki's estimate (\ref{Spec5}) on the energy landscape. 
Furthermore denote by 
\[
 A_\xi:= \left\{u \in H^{N,\varepsilon}\colon u=m_\xi + v \colon  \langle v, m_\xi' \rangle = 0   \text{  and } \| v \|_{H^1(\RR)} \leq \delta \right\}.
\]
Using Funaki's estimate  (\ref{Spec5}) for $u=m_\xi +   v \in A_2$  one can write
\[
 \exp \Bigl(-\frac{1}{\varepsilon} \mathcal{H}(u)\Bigl) \geq \exp \left(-\frac{c_4}{\varepsilon}  \|v\|_{H^1(\RR)}^2  \right).
\]
Note the $v$ is not an element of the discretized space $H^{N,\varepsilon}$ but a general function in $L^2(\RR)$ that needs not vanish outside of $[-\varepsilon^{-\gamma},\varepsilon^{-\gamma}]$. But $v$ can be well approximated by a function $v^{N,\varepsilon}=u-m_\xi^{N,\varepsilon} \in H^{N,\varepsilon}_0$. In fact using Lemma \ref{discbo1} one gets
\[
\| v^{N,\varepsilon}- v \|_{H^1(\RR)}=\| m_\xi^{N,\varepsilon}-m_\xi \|_{H^1(\RR)} \leq C  \frac{\varepsilon^{-\gamma}}{N}  \varepsilon^{\frac{-\gamma_1}{2}}.
\]
Putting this together one gets:
\begin{equation}\label{LoBo2}
\begin{split}
Z^{N,\varepsilon} &{Z^{N,\varepsilon}_1}\exp\Bigl(\frac{C_*}{\varepsilon}\Bigr)  \geq   \int_A  \exp \left(-\frac{c_4}{\varepsilon} \|v\|_{H^1(\RR)}^2   \right)\mathcal{L}^{N,\varepsilon}(\mathrm{d}u)\\
&\geq \exp \left( -2 C \left( \frac{\varepsilon^{-\gamma-\gamma_1/2}}{\varepsilon N} \right) \right) \int_{A} \exp \left(-\frac{2 c_4}{\varepsilon}  \|v^{N,\varepsilon} \|_{H^1(\RR)}^2 \right)   \mathcal{L}^{N,\varepsilon}(\mathrm{d}u).
\end{split}
\end{equation}
Let us concentrate on the integral term in equation \ref{LoBo2}. Using the coarea formula \ref{Coarea} one gets: 
\begin{equation}\label{LoBo3}
\int_{A} \exp \left(-\frac{2 c_4}{\varepsilon}  \|v^{N,\varepsilon} \|_{H^1}^2   \right) \mathcal{L}^{N,\varepsilon}(\mathrm{d}u) \geq \int_{I_\varepsilon} \mathrm{d}\xi \int_{A_\xi} \frac{1}{|\tilde{\nabla}\tilde{ f}|}\exp \Bigl(-\frac{c_4}{\varepsilon} \|v\|_{H^1}^2 \Bigl) \mathcal{H}^{N,\varepsilon}(\mathrm{d}u).
\end{equation}
where $\mathcal{H}^{N,\varepsilon}$ is the codimension one Hausdorff measure on $H^{N,\varepsilon}$.  Using Lemma \ref{Frech} and the observation from Lemma \ref{fdid} one knows:
\[
 \frac{1}{|\tilde{\nabla}\tilde{ f}|}\geq \frac{1}{2} \sqrt{\frac{N}{\varepsilon^{-\gamma}}}   \frac{|m_\xi'|_{L^2(\RR)}^2 + \langle v, m_\xi'' \rangle_{L^2(\RR)} }{\| m_\xi' \|_{L^2(\RR)}}.
\]
By choosing a smaller $\delta$ if necessary this can be bounded uniformly from below on $A$ by a $C \sqrt{\frac{N}{\varepsilon^{-\gamma}}} $ such that one gets:
\begin{equation}\label{LoBo9}
\int_{A} \exp \left(-\frac{2 c_4}{\varepsilon}  \|v^{N,\varepsilon} \|_{H^1(\RR)}^2  \right) \mathcal{L}^{N,\varepsilon}(\mathrm{d}u) \geq     C \sqrt{\frac{N}{\varepsilon^{-\gamma}}} \int_{I_\varepsilon} \mathrm{d} \xi  \int_{A_\xi} \exp \left(-\frac{2 c_4}{\varepsilon} \|v^{N,\varepsilon} \|_{H^1(\RR)}^2\right)  \mathcal{H}^{N,\varepsilon}(\mathrm{d}u). 
\end{equation}
Let us focus on the last integral. By a linear change of coordinates one can write 
\begin{equation}\label{LoBo4}
 \int_{A_\xi} \exp \left(-\frac{2 c_4}{\varepsilon} \|v^{N,\varepsilon} \|_{H^1(\RR)}^2\right)  \mathcal{H}^{N,\varepsilon}(\mathrm{d}u)=\int_{B_\xi} \exp \left(-\frac{2 c_4}{\varepsilon}  \|v \|_{H^1(\RR)}^2\right)  \mathcal{H}^{N,\varepsilon}(\mathrm{d}v),
\end{equation}
where $B_\xi= \left\{v \in H^{N,\varepsilon}_0 \colon  \langle v, m_\xi' \rangle_{L^2(\RR)} = \langle m_\xi-m_\xi^{N,\varepsilon}, m_\xi'  \rangle_{L^2(\RR)}  \text{  and } \| v \|_{H^1(\RR)} \leq \delta \right\}$. 
In order to conclude, we need the following elementary lemma:
\begin{lemme}\label{Linalg2}
Let $E$ be a finite dimensional Euclidean space with Lebesgue measure $\mathcal{L}$ and codimension $1$ Hausdorff measure $\mathcal{H}$. Let $a^*= \langle a, \cdot \rangle \in E^*$ be a linear form and  $x \mapsto \langle x, \Sigma x \rangle$ be a symmetric, positive bilinear form. Furthermore write for $ b \in \RR$
\[
 \tilde{B}_b= \left\{ x \in E \colon ax=b \text{ and } \langle x, \Sigma x \rangle \leq \delta^2 \right\}. 
\]
Furthermore set $d^2= \inf_{x \in \tilde{B}_b} \langle x, \Sigma x  \rangle$ and let $n$ be a $\Sigma$-unit normal vector on $\tilde{B}_0$, i.e. $\langle n, \Sigma x \rangle =0$ for all $x \in \tilde{B}_0$ and $\langle n ,\Sigma n \rangle =1$. Then one has for every $b$
\begin{equation}\label{Linalg}
 \int_{\langle x, \Sigma x \rangle \leq \delta^2} \exp \left( - \langle x,\Sigma x \rangle \right) \mathcal{L}(\mathrm{d}x) \leq  2 \delta \sqrt{ \frac{1}{\langle n,n \rangle}}  \exp \left( d^2 \right)  \int_{\tilde{B}_b} \exp \left(- \langle x,\Sigma x \rangle \right) \mathcal{H}(\mathrm{d}x).
\end{equation}
Furthermore $d$ and $\langle n, n \rangle $ can be given as follows:
\begin{equation}\label{varpri1}
 d^2=\frac{b^2}{\langle a, \Sigma^{-1} a \rangle } \qquad \text{and} \quad  \langle a, \Sigma^{-1} a \rangle= \sup_{\eta \colon \langle \eta ,\Sigma \eta \rangle =1} \langle a ,\eta \rangle,
\end{equation}
and 
\[
 \left(\langle n,n \rangle \right)^{1/2}= \sup_{\eta \colon \langle \eta ,\Sigma \eta \rangle =1} \langle n ,\eta \rangle.
\]

\end{lemme}
\begin{proof}(Of Lemma \ref{Linalg2}):
Using the Coarea formula one can write:  
\begin{equation}\label{Linalg3}
 \begin{split}
\int_{\langle x, \Sigma x \rangle \leq 2 \delta^2-d^2} \exp \left( -\langle x,\Sigma x \rangle \right) \mathcal{L}&(\mathrm{d}x) \leq  \int_{-\delta}^{\delta} \int_{ \tilde{B}_0}  \exp \left(- \langle (y+\lambda n),\Sigma (y+\lambda n) \rangle \right) \sqrt{\frac{1}{\langle n,n \rangle}} \mathcal{H}(\mathrm{d}y) \mathrm{d}\lambda\\
&\leq \sqrt{ \frac{1}{\langle n,n \rangle}}  \int_{-\delta}^{\delta} \int_{ \tilde{B}_0}  \exp \left(- \langle y,\Sigma y \rangle \right) \mathcal{H}(\mathrm{d}y) \mathrm{d}\lambda\\
&=2 \delta \sqrt{ \frac{1}{\langle n,n \rangle}}  \int_{ \tilde{B}_0}  \exp \left(- \langle y,\Sigma y \rangle \right) \mathcal{H}(\mathrm{d}y)\\
&= 2 \delta \sqrt{ \frac{1}{\langle n,n \rangle}} \exp \left( d^2 \right)  \int_{ \tilde{B}_0} \exp \left(- \langle (y +d n),\Sigma (y + d n) \rangle \right) \mathcal{H}(\mathrm{d}y)\\
&= 2 \delta \sqrt{ \frac{1}{\langle n,n \rangle}} \exp \left( d^2 \right)  \int_{ \tilde{B}_b} \exp \left(- \langle y ,\Sigma y  \rangle \right) \mathcal{H}(\mathrm{d}y).
\end{split}
\end{equation}
The other assertions are elementary.
\end{proof}
In order to apply this Lemma to the case $E = H^{N,\varepsilon}$, $a^*(v) = \langle v, m_\xi' \rangle_{L^2(\RR)}$ $b= \langle m_\xi-m_\xi^{N,\varepsilon}, m_\xi'  \rangle_{L^2(\RR)}$ and $\langle v,\Sigma v \rangle= \| v \|_{H^1(\RR)}^2$ one needs to evaluate the constants $d$ and $\langle n,n \rangle$ in this context. This is subject of the next Lemma:
\begin{lemme}\label{LinA}
One has:
\begin{enumerate}
 \item[(i)] $\langle m_\xi-m_\xi^{N,\varepsilon}, m_\xi'  \rangle_{L^2(\RR)} \leq C \varepsilon^{-\gamma_1} \frac{\varepsilon^{-2\gamma}}{N^2} $,
 \item[(ii)] $d^2 \leq C \varepsilon^{-\gamma_1} \left( \frac{\varepsilon^{-\gamma}}{N} \right)^{1/2}   $,
 \item[(iii)] $\langle n, n \rangle = \langle a, \Sigma^{-1} a \rangle \geq C \sqrt{ \frac{\varepsilon^{-\gamma}}{N}} \frac{\varepsilon^{-\gamma}}{N}$.
\end{enumerate}
\end{lemme}
\begin{proof}(Of Lemma \ref{LinA})
(i) Applying Cauchy-Schwarz inequality one gets
\[
\langle m_\xi-m_\xi^{N,\varepsilon}, m_\xi'  \rangle_{L^2(\RR)}  \leq \| m_\xi-m_\xi^{N,\varepsilon} \|_{L^2(\RR)} \| m_\xi' \|_{L^2} \leq C \varepsilon^{-\gamma_1/2} \frac{\varepsilon^{-2\gamma}}{N^2}.
\]
Here Lemma \ref{discbo1} was used.
(ii) In order to evaluate $d$ note first that the Euclidean coordinates $a_k \quad k=-(N-1), \ldots , (N-1)$ of the vector associated to the linear form $a^*$ are given as
\[
 a_k=\langle e_k, m_\xi' \rangle_{L^2(\RR)}.
\]
Here the hat functions $e_k$ are defined like in (\ref{Hat}). In order to get a lower bound on $\langle a, \Sigma a \rangle$ we use the variational principle given in (\ref{varpri1}). Choose as a testfunction $\eta= \hat{\eta} e_k$. One has
\[
 \| e_k \|_{H^1(\RR)}^2=\frac{2\varepsilon^{-\gamma}}{3N}+2\frac{N}{\varepsilon^{-\gamma}} ,
\]
such that one has to chose $\hat{\eta}= \sqrt{\frac{3 \frac{\varepsilon^{-\gamma}}{N}    }{2\frac{\varepsilon^{-2\gamma}}{N^2}+3} }  $ in order to guarantee that $\| \eta \|_{H^1(\RR)}=1$. Note that as 
$\frac{\varepsilon^{-\gamma}}{N} \downarrow 0$ for $\varepsilon \downarrow 0$ one can bound $\xi$ uniformly from below by $C \sqrt{ \frac{\varepsilon^{-\gamma}}{N}}$. Now set $c_7 = \inf_{s \in [-2,2]} m_\xi'(s) >0$ and chose $k$ such that $[s_{k-1}^{N,\varepsilon}, s_{k+1}^{N,\varepsilon}] \subseteq [\xi-2,\xi+2]$, which is always possible for $\varepsilon$ small enough. Then one gets 
\[
 \langle \eta, a \rangle = \xi \langle e_k, m_\xi' \rangle \geq c_7 \xi \| e_k \|_{L^1(\RR)} \geq C \sqrt{ \frac{\varepsilon^{-\gamma}}{N}} \frac{\varepsilon^{-\gamma}}{N}.
\]
Therefore using (i) one gets 
\[
 d^2 \leq C \varepsilon^{-\gamma_1} \left( \frac{\varepsilon^{-\gamma}}{N} \right)^{1/2}. 
\]
The statement (iii) follows immediately. 
\end{proof}
\emph{ End of proof of Proposition \ref{LoBo1}:}
Applying Lemma \ref{Linalg2} and \ref{LinA} to equation (\ref{LoBo4}) one gets:
\begin{equation}\label{LoBo5}
\begin{split}
 \int_{B_\xi} &\exp \left(-\frac{2 c_2}{\varepsilon}  \|v \|_{H^1(\RR)}^2 \right) \mathcal{H}^{N,\varepsilon}(\mathrm{d}v)\\
 &\geq \frac{1}{\delta}    \int_B  C \sqrt{ \frac{\varepsilon^{-\gamma}}{N}} \frac{\varepsilon^{-\gamma}}{N}   \exp\left(-C \varepsilon^{-\gamma_1} \left( \frac{\varepsilon^{-\gamma}}{N} \right)^{1/2}  \right)    \exp \left(-\frac{2 c_2}{\varepsilon} \|v \|_{H^1}^2 \right) \mathcal{L}(\mathrm{d}v)\\
&=   \sqrt{ \frac{\varepsilon^{-\gamma}}{N}} \frac{\varepsilon^{-\gamma}}{N}   \exp\left(-C \varepsilon^{-\gamma_1} \left( \frac{\varepsilon^{-\gamma}}{N} \right)^{1/2}  \right)  Z^{N,\varepsilon}_2 \sigma \left( \|v \|_{H^1(\RR)}^2 \leq \delta \right).
\end{split}
\end{equation}
where $B= \left\{v \in H^{N,\varepsilon}_0 \text{ s. th. } \| v \|_{H^1} \leq \delta  \right\}$.  Recall that $\sigma$ is the Gaussian measure discussed in Lemma \ref{HGM}. According to Lemma \ref{HGM} for $\varepsilon$ small enough $\sigma \left( \|v \|_{H^1}^2\| \leq \delta \right) \geq \frac{1}{2}$. Therefore the following lemma concludes the proof.
\end{proof}

\begin{lemme}
The Gaussian normalization constants $Z^{N,\varepsilon}_1$ and $Z^{N,\varepsilon}_2$ satisfy the following:
\begin{equation}\label{GaNo}
c_1^{-N}  \left(1+\frac{2\varepsilon^{-\gamma}}{\pi} \right)^{N}  \leq \frac{Z^{N,\varepsilon}_2}{Z^{N,\varepsilon}_1} \leq c_1^{-N}.
\end{equation}
 \end{lemme}
\begin{proof}
By definition 
\[
 Z^{N, \varepsilon}_2=(2 \pi)^N \varepsilon^{N} c_4^{-N} \det(1-\Delta_{N,\varepsilon})^{-1/2},
\]
where $\Delta_{N,\varepsilon}$ is the Dirichlet Laplacian on $[-\varepsilon^{-\gamma}, \varepsilon^{-\gamma}]$ restricted to $H^{N,\varepsilon}$ and
\[
 Z^{N, \varepsilon}_1=(2 \pi)^N  \varepsilon^{N}\det(-\Delta_{N,\varepsilon})^{-1/2}.
\]
By Poincar\'e inequality one has 
\[
 -\Delta_{N,\varepsilon} \leq \left( 1-\Delta_{N,\varepsilon} \right) \leq \left(1+\frac{2\varepsilon^{-\gamma}}{\pi}\right) (-\Delta_{N,\varepsilon}),
\]
in the sense of selfadjoint operators. This implies
\[
\det(-\Delta_{N,\varepsilon}) \leq \det(1-\Delta_{N,\varepsilon}) \leq \left( 1+\frac{2\varepsilon^{-\gamma}}{\pi} \right)^{2N}\det(-\Delta_{N,\varepsilon}),
\]
and therefore
\[
c_4^{-N}  \left( 1+\frac{2\varepsilon^{-\gamma}}{\pi} \right)^{N}  \leq \frac{Z^{N,\varepsilon}_2}{Z^{N,\varepsilon}_1} \leq c_4^{-N}.
\]
\end{proof}

\vskip3ex

As a next step an upper bound on $\mu^{N,\varepsilon}(A^c)$ is derived:
\begin{prop}\label{UpBo}
Choosing $\gamma_1$ and $\gamma_2$ as in (\ref{gamma1}),(\ref{gamma2}) and (\ref{gamma3}) one has for $\delta \leq \delta_0$:
\begin{equation}\label{Upbo1}
\limsup_{\varepsilon \downarrow 0} \varepsilon \log \Bigl( Z^{N,\varepsilon} \mu^{N,\varepsilon}\left( \text{\emph{dist}}_{H^1}(u,M) \geq \delta \right)\Bigr) \leq C_*+c_0 \delta^2.
\end{equation}
 \end{prop}
\begin{proof}
Denote by $A^\delta := \left\{u \colon \text{dist}_{H^1}(u,M) \geq \delta   \right\}$. Then one has 
\begin{equation}\label{UpBo2}
\begin{split}
 Z^{N,\varepsilon} \mu^{N,\varepsilon}(A^\delta) &= \exp \Bigl(-\frac{C_*}{\varepsilon} \Bigr)\frac{1}{Z^{N,\varepsilon}_1} \int_{A^\delta} \exp\Bigl(- \frac{1}{\varepsilon}\mathcal{H}(u) \Bigr)  \lambda^{N,\varepsilon}(\mathrm{d}u) \\
& \leq \exp \Bigl(-\frac{C_*+c_0 \delta^2}{\varepsilon} \Bigr)\frac{1}{Z^{N,\varepsilon}_1} \int_{A^\delta} \exp\Bigl(- \frac{1}{\varepsilon}\left( \mathcal{H}(u)-c_0\delta^2 \right) \Bigr)  \lambda^{N,\varepsilon}(\mathrm{d}u).
\end{split}
\end{equation}
Note that by Lemma \ref{GloBo} $\mathcal{H}(u)-c_0 \delta^2 \geq 0$ on $A^\delta$. So on this set one gets 
\[
\exp \Bigl(- \frac{1}{\varepsilon}\left( \mathcal{H}(u)-c_0\delta^2  \right) \Bigr) \leq \exp \Bigl(- \left( \mathcal{H}(u)-c_0\delta^2 \right) \Bigr).
\]
Therefore one gets
\begin{equation}\label{UpBo3}
\begin{split}
\int_{ A^\delta}& \exp \Bigl(- \frac{1}{\varepsilon}\left( \mathcal{H}(u)-c_0 \delta^2 \right) \Bigr)  \lambda^{N,\varepsilon}(\mathrm{d}u)\leq  \int_{ A^\delta} \exp \Bigl(- \left( \mathcal{H}(u)-c_0 \delta^2 \right) \Bigr) \lambda^{N,\varepsilon}(\mathrm{d}u) \\
\leq &\int_{ A^\delta} Z^{N,\varepsilon}_3 \exp \bigg(\int_{-\varepsilon^{-\gamma}}^{\varepsilon^{-\gamma}} - F\big(u(s) \big) \, \mathrm{ds} + c_0 \delta^2 \bigg) \, \nu^{1,N}(\mathrm{d}u), 
\end{split}
\end{equation}
where $\nu^{1,N}$ is the discretized Brownian Bridge without rescaling and 
\[
Z^{N,\varepsilon}_3=\int \exp\Bigl(- \int_{-\varepsilon^{-\gamma}}^{\varepsilon^{-\gamma}} \frac{1}{2} |u'(s)|^2  \, \mathrm{ds} \Bigr) \, \mathcal{H}^{N,\varepsilon}( \mathrm{d}u) 
\]
is the appropriate normalization constant. Using the positivity of $F$  the last term in (\ref{UpBo3}) can therefore be bounded by
\[
Z^{N,\varepsilon}_3 \exp \Bigl(c_0 \delta^2 \Bigr).
\]
Plugging this into (\ref{UpBo2}) yields
\[
 Z^{N,\varepsilon} \mu^{N,\varepsilon}(A^\delta) \leq \exp \Bigl(-\frac{C_*+c_0 \delta^2}{\varepsilon} \Bigr)\frac{1}{Z^{N,\varepsilon}_1}  Z^{N,\varepsilon}_3 \exp \Bigl(c_0 \delta^2 \Bigr).
\]
This finishes the proof together with the following bound on the normalization constants $Z^{N,\varepsilon}_1$ and $Z^{N,\varepsilon}_3$.
\end{proof}
\begin{lemme}
One has
\[
 \frac{Z^{N,\varepsilon}_3}{Z^{N,\varepsilon}_1}=\varepsilon^{-N}.
\]
\end{lemme}
\begin{proof}
 This is a direct consequence of the fact that for matrices $A \in \RR^{n \times n}$ and $\xi \in \RR$ 
\[
 \det(\xi A)= \xi^n \det(A),
\]
as well as the explicit formula for the Gaussian normalization constants.
\end{proof}
One can now summarize the finite dimensional calculation in the following:
\begin{cor}\label{fidica}
Choosing the constants $\gamma_1$ and $\gamma_2$ as in (\ref{gamma1}),(\ref{gamma2}),(\ref{gamma3}) one obtains for $\delta \leq \delta_0$:
\[ 
\limsup_{\varepsilon \downarrow 0} \varepsilon \log \left(  \mu^{N,\varepsilon}(\text{\emph{dist}}_{H^1}(u,M) \geq \delta)\right) \leq c_0 \delta^2.
\]
Note that such a choice is possible for all $\gamma <1$.
\end{cor}
\begin{proof}
Dividing and using the estimates from above yields the result.
\end{proof}

Using again the continuous embedding of $H^1$ into $L^\infty$ one gets:
\begin{cor}\label{fidica2}
Choosing the constants $\gamma_1$ and $\gamma_2$ as in (\ref{gamma1}),(\ref{gamma2}),(\ref{gamma3}) one obtains for $\delta \leq \delta_0$:
\[ 
\limsup_{\varepsilon \downarrow 0} \varepsilon \log \left(  \mu^{N,\varepsilon}(\text{\emph{dist}}_{L^\infty}(u,M) \geq \delta)\right) \leq c_0 \delta^2.
\]
Such a choice is possible for all $\gamma <1$.
\end{cor}

As a last step in this section we need to control the deviations from the discretized measure with the help of the Gaussian estimates derived in the last section. To this end one has to estimate the deviations of the normalization constant $Z^\varepsilon$ from $Z^{N,\varepsilon}$. In order to proof the following Lemma we will need an additional assumption on the double well potential $F$. 

\vskip4ex
{\bf Assumption:}
\begin{equation}\label{AssA}
|F'(u)| \text{ is bounded for $u \in \RR$ }. 
\end{equation}
\vskip4ex
In fact one can simply modify the potential $F$ by cutting it off outside of some compact set, such that it satisfies (\ref{AssA}). We will proceed now by proving Theorem \ref{THM2} under the additional assumption (\ref{AssA}). The general case will then follow as a Corollary.

\begin{prop}\label{normconst}
Assume that $F$ satisfies (\ref{AssA}). Furthermore suppose $\gamma < \frac{2}{3}$. Then one has the following bound:
\begin{equation}\label{NoCoLoBo}
 \liminf_{\varepsilon \downarrow 0} \varepsilon \log Z^{\varepsilon} \geq - C_*.
\end{equation}
 \end{prop}
\begin{proof}
Denote as above by $u^N$ the piecewise linearization of the function $u$. Note that we work with the continuous version of $u$ such that this is an a.s. well defined operation. Then one can write:
\begin{equation*}
\begin{split}
 Z^\varepsilon &= \int_{H^{\varepsilon}} \exp \left(-\frac{1}{\varepsilon}\int_{-\varepsilon^{-\gamma}}^{\varepsilon^{-\gamma}}F(u(s))\mathrm{ds}  \right) \nu^\varepsilon(\mathrm{d}u)\\
&= \int_{H^{\varepsilon}} \exp \left( -\frac{1}{\varepsilon}\int_{-\varepsilon^{-\gamma}}^{\varepsilon^{-\gamma}}F(u^N(s))\mathrm{ds}  \right) \exp \left( -\frac{1}{\varepsilon} \int_{-\varepsilon^{-\gamma}}^{\varepsilon^{-\gamma}}\left( F(u(s))-F(u^N(s))\right) \mathrm{ds}  \right) \nu^\varepsilon(\mathrm{d}u)\\
&\geq \int_{H^{\varepsilon}} \exp \left( -\frac{1}{\varepsilon}\int_{-\varepsilon^{-\gamma}}^{\varepsilon^{-\gamma}}F(u^N(s))\mathrm{ds}  \right)  \exp \left( -\frac{1}{\varepsilon} \| F' \|_\infty  \sum_{k=-N}^{(N-1)}\int_{s_k^{N,\varepsilon}}^{s_{k+1}^{N,\varepsilon}}\left| u(s)-u^N(s)\right|\mathrm{ds}   \right)  \nu^\varepsilon(\mathrm{d}u)\\
&\geq \int_{H^{\varepsilon}} \exp \left( -\frac{1}{\varepsilon}\int_{-\varepsilon^{-\gamma}}^{\varepsilon^{-\gamma}}F(u^N(s))\mathrm{ds}  \right) \\
& \hspace{2cm} \exp \left( -\frac{1}{\varepsilon} \| F' \|_\infty   \left( \frac{\varepsilon^{-\gamma}}{N} \right)^{1/2} \sum_{k=-N}^{(N-1)}  \left(   \int_{s_k^{N,\varepsilon}}^{s_{k+1}^{N,\varepsilon}}\left| u(s)-u^N(s)\right|^2 \mathrm{ds} \right)^{1/2}  \right)  \nu^\varepsilon(\mathrm{d}u).
\end{split}
\end{equation*}
Now one can use the independence of the discretized Brownian bridge and the intermediate bridges to write the last term as:
\begin{equation}\label{noco3}
\begin{split}
 Z^{N,\varepsilon} \prod_{k=-N}^{(N-1)} \int_{H^{\varepsilon}} \exp \left( -\frac{1}{\varepsilon} \| F' \|_\infty   \left( \frac{\varepsilon^{-\gamma}}{N} \right)^{1/2}   \left(   \int_{s_k^{N,\varepsilon}}^{s_{k+1}^{N,\varepsilon}}\left| u(s)-u^N(s)\right|^2 \mathrm{ds} \right)^{1/2}  \right)  \nu^\varepsilon(\mathrm{d}u). 
\end{split}
\end{equation}
Let us calculate the integrals: Using the formula
\[
 \EE[e^{-\beta x}]= 1-\beta \int_0^\infty e^{-\beta x} \PP \left[X \geq x \right] \mathrm{d}x,
\]
which holds for every non-negative random variable and every $\beta> 0$ one obtains:
\begin{equation}\label{noco1}
\begin{split}
\int_{H^{\varepsilon}} &\exp \left( -\frac{1}{\varepsilon} \| F' \|_\infty   \left( \frac{\varepsilon^{-\gamma}}{N} \right)^{1/2}   \left(   \int_{s_k^{N,\varepsilon}}^{s_{k+1}^{N,\varepsilon}}\left| u(s)-u^N(s)\right|^2 \mathrm{ds} \right)^{1/2}  \right)  \nu^\varepsilon(\mathrm{d}u)=\\
 & 1- \frac{1}{\varepsilon} \| F' \|_\infty   \left( \frac{\varepsilon^{-\gamma}}{N} \right)^{1/2}  \int_0^\infty  \exp \left( -\frac{1}{\varepsilon} \| F' \|_\infty   \left( \frac{\varepsilon^{-\gamma}}{N} \right)^{1/2}    x  \right)  \times \\
 &  \hspace{5cm} \times \nu^\varepsilon \left( \left( \int_{s_k^{N,\varepsilon}}^{s_{k+1}^{N,\varepsilon}} \left| u(s)-u^N(s)\right|^2 \mathrm{ds}  \right)^{1/2}   \geq x   \right)\mathrm{d}x.
\end{split}
\end{equation}
Using the inequality (\ref{Disbound2}) one can bound the term in (\ref{noco1}) from below by
\begin{equation}\label{noco2}
\begin{split}
1&- \frac{1}{\varepsilon} \| F' \|_\infty   \left( \frac{\varepsilon^{-\gamma}}{N} \right)^{1/2}  \int_0^{\sqrt{\varepsilon \frac{\varepsilon^{-2\gamma}}{6 N^2}}} \exp \left( -\frac{1}{\varepsilon} \| F' \|_\infty   \left( \frac{\varepsilon^{-\gamma}}{N} \right)^{1/2}    x  \right)\mathrm{d}x  \\
 & - \frac{1}{\varepsilon} \| F' \|_\infty   \left( \frac{\varepsilon^{-\gamma}}{N} \right)^{1/2} 
\exp \left( -\frac{1}{\varepsilon} \| F' \|_\infty   \left( \frac{\varepsilon^{-\gamma}}{N} \right)^{1/2}  \sqrt{\varepsilon \frac{\varepsilon^{-2\gamma}}{6 N^2}} \right) \times \\               &\times \int_{0}^\infty \exp \left( -\frac{1}{\varepsilon} \| F' \|_\infty   \left( \frac{\varepsilon^{-\gamma}}{N} \right)^{1/2}    x  \right)\exp \left(\frac{-x^2 \pi^2 N^2}{\varepsilon^{1-2\gamma}} \right) \mathrm{d}x.
\end{split}
\end{equation}
The second term in (\ref{noco2}) yields:
\[
  1- \exp \left( -\frac{1}{\varepsilon} \| F' \|_\infty   \left( \frac{\varepsilon^{-\gamma}}{N} \right)^{1/2}  \sqrt{\varepsilon \frac{\varepsilon^{-2\gamma}}{6 N^2}}    \right).
\]
Using the elementary inequality
\[
 \int_0^\infty e^{-\alpha x^2 - \beta x} dx \leq \frac{1}{\beta}\exp \left(-\frac{\beta}{4} \left(  \frac{\beta}{\alpha}+2  \right)  \right),
\]
for $\alpha,\beta >0$ which can be obtained by completing the squares and applying the standard estimate $\int_\gamma^\infty e^{-\frac{x^2}{2}}dx \leq \frac{1}{\gamma}e^{-\frac{\gamma^2}{2}}$ one can bound the third term by:
\begin{equation*}
 \begin{split}
\exp &\left( -\frac{1}{\varepsilon} \| F' \|_\infty   \left( \frac{\varepsilon^{-\gamma}}{N} \right)^{1/2}  \sqrt{\varepsilon \frac{\varepsilon^{-2\gamma}}{6 N^2}} \right) \times \\
 & \hspace{2cm} \times \exp\left(- \frac{1}{ 4\varepsilon} \| F' \|_\infty   \left( \frac{\varepsilon^{-\gamma}}{N} \right)^{1/2} \left( \frac{1}{ 4\varepsilon} \| F' \|_\infty   \left( \frac{\varepsilon^{-\gamma}}{N} \right)^{1/2}  \left(  \frac{ \pi^2 N^2}{\varepsilon^{1-2\gamma}}  \right)^{-1}         +2 \right)    \right) .
 \end{split}
\end{equation*}
Noting that the last exponential converges to zero as $\varepsilon \downarrow 0$ it can in particular be bounded by $\frac{1}{2}$ such that  in total the expression in (\ref{noco3}) can be bounded from below by:
\[
\frac{1}{2} Z^{N,\varepsilon} \exp \left( -\frac{1}{\varepsilon} \| F' \|_\infty N  \left( \frac{\varepsilon^{-\gamma}}{N} \right)^{1/2}  \sqrt{\varepsilon \frac{\varepsilon^{-2\gamma}}{6 N^2}}    \right).
\]
In particular the exponent scales like
\[
 \varepsilon^{-1-\frac{3\gamma}{2}+\frac{1}{2}}N^{-\frac{1}{2}}.
\]
so by choosing $\gamma_1$ such that $-3 \gamma +1 +\gamma_1 >0$  one obtains the desired result together with Lemma \ref{LoBo5}. Note that such a choice is possible for every $\gamma < \frac{2}{3}$. 
\end{proof}
Now one can conclude
\begin{prop}\label{modThm2}
The statement of Theorem \ref {THM2} and Theorem \ref {THM3} hold under the additional assumption (\ref{AssA}). 
\end{prop}

\begin{proof}
Let $\lambda^\varepsilon(\mathrm{d}u,\mathrm{d}v)$ be the joint distribution of the rescaled Brownian bridge on $[-\varepsilon^{-\gamma}, \varepsilon^{-\gamma}]$ and its discretization. In particular $\lambda^\varepsilon$ is a coupling of $\nu^\varepsilon$ and $\nu^{N,\varepsilon}$. We had seen above in Lemma \ref{DiscBound}, that 
\begin{equation}\label{pro1}
 \lambda \Bigl(\|u - v\|_{L^2(\RR)}\geq \delta/2 \Bigr) \leq \exp \left(-\frac{r^2 \pi^2 N^2}{\varepsilon^{1-2\gamma}} \right).
\end{equation}
Define a new measure $\lambda_1$ on $E \times E$ by
\[
 \lambda_1(\mathrm{d}u,\mathrm{d}v)= \frac{1}{Z^{N,\varepsilon}}\frac{1}{Z^{\varepsilon}} \exp\bigg( - \frac{1}{\varepsilon} \int F\big( u(s) \big ) \, \mathrm{ds} \bigg)  \, \exp\bigg( - \frac{1}{\varepsilon} \int F\big(v(s)\big) \, \mathrm{ds} \bigg) \, \lambda(\mathrm{d}u,\mathrm{d}v).
\]
The measure $\lambda_1$ is a coupling of $\mu^\varepsilon$ and  $\mu^{N,\varepsilon}$.
Then one can estimate
\begin{equation}
\begin{split}
& \mu^\varepsilon\Bigl(\text{dist}_{L^2}(u,M)\geq \delta    \Bigr)= \lambda_1^\varepsilon\Bigl(\text{dist}_{L^2}(u,M)\geq \delta    \Bigr)\\
& \qquad \leq  \lambda_1^\varepsilon\Bigl(\text{dist}_{L^2}(u,M)\geq \delta  ; \| v-u \|_{L^2(\RR)} \geq \delta  \Bigr) + \lambda_1^\varepsilon\Bigl(\text{dist}_{L^2}(u,M)\geq \delta  ; \| v-u \|_{L^2(\RR)} \leq \delta  \Bigr)\\
&\qquad =I_1+I_2.
\end{split}
\end{equation}
The second term $I_2$ can be estimated
\[
 I_2 \leq \nu^{N,\varepsilon}(\text{dist}_{L^2}(u,m) \geq 2 \delta),
\]
which can be bounded using (\ref{fidica}). The first term can be bounded by
\[
I_1 \leq  \lambda_1^\varepsilon \Bigl( \| v-u \|_{L^2(\RR)} \geq \delta    \Bigr) \leq \frac{1}{Z^{N,\varepsilon}}\frac{1}{Z^{\varepsilon}} \lambda^\varepsilon  \Bigl( \| v-u \|_{L^2(\RR)} \geq \delta    \Bigr),
\]
which converges to zero by (\ref{pro1}) as well as Lemma \ref{LoBo5} together with Lemma \ref{NoCoLoBo}. Note that for this one needs $\gamma_2 > \gamma$. This finishes the proof for the $L^ 2$-norm. To the see analogue result for the $L^\infty$-norm repeat the same reasoning with (\ref{fidica}) replaced by (\ref{fidica2}) and the $L^2$ bound (\ref{Disbound2}) replaced by the $L^\infty$-bound (\ref{Disbound3}).   
\end{proof}
\begin{proof}(Of Theorem \ref{THM2} and \ref{THM3} in the general case): Denote by $\text{dist}$ either $\text{dist}_{L^2}$ or $\text{dist}_{L^\infty}$. Assume that $F$ only satisfies assumptions (\ref{assumptions}). By cutting $F$ off outside of $[-2,2]$ one can chose a function $\bar{F}$ that coincides with $F$ on $[-2,2]$ that satisfies (\ref{assumptions}) and (\ref{AssA}) as well as
\[
\bar{F}(u) \leq  F(u) \qquad \text{for } u \in \RR.
\]
Then one can write
\begin{equation}
 \begin{split}
\mu^\varepsilon\left( \text{dist}(u,M) \geq \delta \right)&= \frac{\int_{\{ \text{dist}(u,M) \geq \delta \} } \exp \Bigl (-\varepsilon^{-1} \int_{-\varepsilon^{-\gamma}}^{\varepsilon^{-\gamma}} F(u(s)) \mathrm{ds} \Bigr) \nu^\varepsilon ( \mathrm{d}u) }{\int    \exp \Bigl (-\varepsilon^{-1} \int_{-\varepsilon^{-\gamma}}^{\varepsilon^{-\gamma}} F(u(s)) \mathrm{ds} \Bigr) \nu^\varepsilon ( \mathrm{d}u)    } \\
&\leq   \frac{\int_{\{ \text{dist}(u,M) \geq \delta \} } \exp \Bigl (-\varepsilon^{-1} \int_{-\varepsilon^{-\gamma}}^{\varepsilon^{-\gamma}} F(u(s)) \mathrm{ds} \Bigr) \nu^\varepsilon ( \mathrm{d}u) }{\int_{\{\|u\|_{L^\infty} \leq 2 \} }    \exp \Bigl (-\varepsilon^{-1} \int_{-\varepsilon^{-\gamma}}^{\varepsilon^{-\gamma}} F(u(s)) \mathrm{ds} \Bigr) \nu^\varepsilon ( \mathrm{d}u)    }
 \end{split}
\end{equation}
The denominator of this fraction coincides with
\[
 \int_{\{\|u\|_{L^\infty(\RR)} \leq 2 \} }    \exp \Bigl (-\varepsilon^{-1} \int_{-\varepsilon^{-\gamma}}^{\varepsilon^{-\gamma}} \bar{F}(u(s)) \mathrm{ds} \Bigr) \nu^\varepsilon ( \mathrm{d}u)
\]
and the nominator is bounded from above by
\[
 \int_{\{ \text{dist}(u,M) \geq \delta \} } \exp \Bigl (-\varepsilon^{-1} \int_{-\varepsilon^{-\gamma}}^{\varepsilon^{-\gamma}} \bar{F}(u(s)) \mathrm{ds} \Bigr) \nu^\varepsilon ( \mathrm{d}u),
\]
such that one can write
\begin{equation}
 \begin{split}
 \mu^\varepsilon\left( \text{dist}(u,M) \geq \delta \right) &\leq   \frac{\int_{\{ \text{dist}(u,M) \geq \delta \} \} }    \exp \Bigl (-\varepsilon^{-1} \int_{-\varepsilon^{-\gamma}}^{\varepsilon^{-\gamma}} \bar{F}(u(s)) \mathrm{ds} \Bigr) \nu^\varepsilon ( \mathrm{d}u)}{\int    \exp \Bigl (-\varepsilon^{-1} \int_{-\varepsilon^{-\gamma}}^{\varepsilon^{-\gamma}} \bar{F}(u(s)) \mathrm{ds} \Bigr) \nu^\varepsilon ( \mathrm{d}u)} \times \\
 & \qquad \times \frac{\int    \exp \Bigl (-\varepsilon^{-1} \int_{-\varepsilon^{-\gamma}}^{\varepsilon^{-\gamma}} \bar{F}(u(s)) \mathrm{ds} \Bigr) \nu^\varepsilon ( \mathrm{d}u)}{\int_{\{\|u\|_{L^\infty} \leq 2 \} }    \exp \Bigl (-\varepsilon^{-1} \int_{-\varepsilon^{-\gamma}}^{\varepsilon^{-\gamma}} \bar{F}(u(s)) \mathrm{ds} \Bigr) \nu^\varepsilon ( \mathrm{d}u)}. 
 \end{split}
\end{equation}
Now applying Proposion \ref{modThm2} shows that the second factor can be bounded by $2 $ for $\varepsilon$ small enough and thus applying Proposion \ref{modThm2} to the first factor yields the desired result. 
\end{proof}

With a similar reasoning one can see that the statement of Proposition \ref{normconst} holds also without assumption (\ref{AssA}):
\begin{cor}
Suppose $\gamma < \frac{2}{3}$. Then one has the following bound:
\begin{equation}\label{NoCoLoBo1}
 \liminf_{\varepsilon \downarrow 0} \varepsilon \log Z^{\varepsilon} \geq - C_*.
\end{equation} 
\end{cor}

\section{Conclusion}\label{SEC5}
This last section is devoted to the proof of Theorem \ref{THM1}. First of all the tightness of the measures $\tilde{\mu}^\varepsilon$ is shown. Then a spatial homogeneity property of the measures $\tilde{\mu}^\varepsilon$ is used to characterize the limit measure $\mu$.

\vskip3ex

\begin{prop}\label{Tight}
The family of measures $\tilde{\mu}^\varepsilon$ is tight. All points of accumulation are concentrated on functions of the type 
\begin{equation}
 \tilde{m}_\xi(s)=-\mathbf{1}_{[-1,\xi]}(s)+ \mathbf{1}_{[\xi,1]}(s).
\end{equation}
\end{prop}
\begin{proof}
Denote by $\tilde{M}= \{ \tilde{m}_\xi \colon \xi \in [-1,1] \}$ and $\text{dist}(\tilde{u},\tilde{M})= \inf_{\xi \in [-1,1]} \|\tilde{u}-\tilde{m}_\xi \|_{L^2[-1,1]} $. Furthermore denote by $m^\varepsilon_\xi(s)=m\left(\frac{s-\xi}{\varepsilon} \right)$. Note that for all $\xi \in [-1,1]$  $m_\xi^\varepsilon$ converges to $\tilde{m}_\xi$ in $L^2$. Now choose $\delta>0$ and $\varepsilon_0$ such that $\| m^\varepsilon_\xi - \tilde{m}_\xi\|_{L^2} \leq \frac{\delta}{2} $ for all $\varepsilon \leq \varepsilon_0$. Then \ref{THM2} implies that 
\begin{equation}\label{conbou}
 \begin{split}
\tilde{\mu} \left(\text{dist}_{L^2} (\tilde{u},\tilde{M}) \geq \delta \right)  & \leq \tilde{\mu} \left(\inf_\xi \| \tilde{u} - m^\varepsilon_\xi \|_{L^2[-1,1]} \geq \frac{\delta}{2} \right) \leq \mu^\varepsilon \left(\text{dist}_{L^2}(T^\varepsilon(\tilde{u},M )\geq \frac{\delta}{2\varepsilon}  \right) \downarrow 0. 
\end{split}
\end{equation}
This is sufficient to show the tightness of the measures $\{ \tilde \mu^\varepsilon \}$. In fact fix a small constant $\kappa>0$. Let us construct a precompact set $K$ such that $\tilde {\mu}^\varepsilon (K^C) \leq \kappa$. For a fixed $N \in \NN$ due to (\ref{conbou}) there exists $\varepsilon^N$ such that for all $\varepsilon \leq \varepsilon^N$  
\[
 \tilde{\mu} \left(\text{dist} (\tilde{u},\tilde{M}) \geq \frac{1}{2N} \right) \leq \frac{\kappa}{2^N}. 
\]
In particular there exist finitely many $\xi_i^N \in [-1,1]$ $i=1, \ldots, i_N$ such that for all  $\varepsilon \leq \varepsilon^N$ 
\[
 \mu^\varepsilon \left( \cup_i B( \tilde{m}_{\xi_i^N},\frac{1}{N} \right) \geq 1-\frac{\kappa}{2^N}. 
\]
Furthermore due to tightness of the measures $(\mu^\varepsilon, \varepsilon \geq \varepsilon_N)$ there exist finitely many balls $\tilde{B}_i^N$ of radius $\frac{1}{N}$ such that for all $\varepsilon \geq \varepsilon_N$ one has 
\[
 \mu^\varepsilon \left( \cup_i B_i\right) \geq 1-\frac{\kappa}{2^N}. 
\]
Set $K^N= \Bigl( \bigcup_i B_i \Bigr)  \cup \left( \bigcup_i B\Bigl( \tilde{m}_{\xi_i^N},\frac{1}{N} \right) \Bigr)$ and $K= \cap_N K^N$. Then $K$ is precompact and for all $\varepsilon$ has measure $\geq 1-\kappa$. This shows tightness. The concentration follows from (\ref{conbou}). 
\end{proof}

\begin{proof}(of Theorem \ref{THM1})
The finite dimensional distributions of the random function $\tilde{u}$ under the measure $\tilde{\mu}^\varepsilon$ are given explicitly as
\begin{equation}\label{findimdis}
\mu^\varepsilon(\tilde{u}(s_1)\in \mathrm{d}x_1 \ldots \tilde{u}(s_n) \in \mathrm{d}x_n)= \frac{P_{s_1+1}(-1,x_1)P_{s_2-s_1}(x_1,x_2) \cdots P_{1-s_n}(x_n,1)}{P_{-1,1}(-1,1)}\mathrm{d}x_1 \ldots \mathrm{d}x_n
\end{equation}
with a transition semigroup $P_t$ that can be given explicitly. (See e.g. \cite{RY99} Proposition 3.1 in \textsection  VIII). Fix an integer $N$ and subdivide $[-1,1]$ in $N$ and set for $k= 1, \ldots ,N-1$ that $s_k^N = \frac{2k}{N}-1$. Fix furthermore a small constant $\delta>0$ and set $A_k^N= \{u \colon u(s_1^N) \in [-1-\delta, -1+\delta], \ldots u(s_k^N) \in [-1-\delta, -1+\delta] , u(s_{k+1}^N) \in [1-\delta, 1+\delta] \ldots u(s_{N-1}^N) \in [1-\delta, 1+\delta]   \}$. Applying the explicit shape for these probabilities given in \ref{findimdis} one sees that for a fixed $N$ all these sets  $A_k^N$ have the same probability. But this property does not pass to the limit under weak convergence of measures on $L^2$. Therefore one has to smear out the random function $\tilde{u}$ around the points $s_k^N$. To this end for a fix $N$ fix a $\hat{\delta} < \frac{1}{2N}$ and consider the random vector whose entries are given as $\hat{u}(s_k^N)= \frac{1}{2\hat{\delta}} \int_{s_k^N-\hat{\delta}}^{s_k^N + \hat{\delta}} u(s)\mathrm{d}s$. Again formula (\ref{conbou}) implies that for fixed $N$ and $\varepsilon$ the quantities 
\begin{align*}
 \mu^\varepsilon \Bigl(\hat{u}(s_1^N) \in [-1-\delta, -1+\delta], \ldots , \hat{u}(s_k^N) \in [-1-\delta, -1+\delta] , \hat{u}(s_{k+1}^N) \in [1-\delta, 1+\delta], \ldots \\
,\hat{u} (s_{N-1}^N) \in [1-\delta, 1+\delta]        \Bigr) 
\end{align*}
coincide for different $k$. This property passes to the limit under weak convergence of $L^2$ valued measures, giving the desired characterization of the distribution of the phase separation point $\xi$.  
\end{proof}

\vskip3ex
{\hskip-4ex \emph{Acknowledgement:}}
\vskip1ex 
\hskip-4ex The author wishes to express his gratitude to Tadahisa Funaki for originally suggesting this problem and for encouraging discussions. He furthermore thanks Felix Otto for helpful discussions and in particular his suggestion for the proof of Proposition \ref{SPEC2}.


\begin{thebibliography}{10}
\bibitem[AR90]{AR90}
{\sc Albeverio, S.; R\"ockner, M.}
\newblock  Stochastic differential equations in infinite dimensions: solutions via Dirichlet forms. 
{\em Probab. Theory Related Fields}  89  (1991),  no. 3, 347--386.



\bibitem[AC79]{AC79}
{\sc Allen, S.M., Cahn, J.W.}
\newblock a microscopic theory for antiphase boundary motion and its application to antiphase domain coarsening,
{\em Acta Metall.} 27 (1979), 1085 --1095.


\bibitem[BMP95]{BMP95}
{\sc Brassesco, S.; De Masi, A.; Presutti, E.}
\newblock Brownian fluctuations of the interface in the $D=1$ Ginzburg-Landau equation with noise.
{\em Ann. Inst. H. Poincar\'e Probab. Statist.} 31 (1995), no. 1, 81--118. 



\bibitem[CP89]{CP89}
{\sc Carr, J.; Pego, R. L.}
\newblock  Metastable patterns in solutions of $u\sb t=\varepsilon\sp 2u\sb {xx}-f(u)$.  
{\em Comm. Pure Appl. Math.}  42  (1989),  no. 5, 523--576. 



\bibitem[Ch04]{Ch04}
{\sc Chen, X.}
\newblock Generation, propagation, and annihilation of metastable patterns.
{\em J. Differential Equations} 206 (2004), no. 2, 399--437. 



\bibitem[dPZ92]{dPZ92}
{\sc Da Prato, G.; Zabczyk, J.}
\newblock  Stochastic equations in infinite dimensions. 
{\em Encyclopedia of Mathematics and its Applications}, 44. Cambridge University Press, Cambridge, 1992.



\bibitem[dPZ96]{dPZ96}
{\sc Da Prato, G.; Zabczyk, J.}
\newblock  Ergodicity for infinite-dimensional systems. 
{\em London Mathematical Society Lecture Note Series}, 229. Cambridge University Press, Cambridge, 1996. xii+339 pp. 



\bibitem[DS89]{DS89}
{\sc Deuschel, J.-D.; Stroock, D.}
\newblock Large deviations. 
{\em Pure and Applied Mathematics}, 137. Academic Press, Inc., Boston, MA, 1989. 


\bibitem[FV03]{FV03}
{\sc Fatkullin, I.;Vanden-Eijnden, E.}
\newblock Coarsening by diffusion-annihilation in a bistable system driven by noise.
{\em Preprint}, 2003.



\bibitem[FINR06]{FINR06}
{\sc Fontes, L. R. G.; Isopi, M.; Newman, C. M.; Ravishankar, K.}
\newblock Coarsening, nucleation, and the marked Brownian web. 
{\em Ann. Inst. H. Poincar\'e Probab. Statist.} 42 (2006), no. 1, 37--60.



\bibitem[Fu95]{Fu95} 
{\sc Funaki, T.}
\newblock The scaling limit for a stochastic PDE and the separation of phases. 
{\em Probab. Theory Related Fields} 102 (1995), no. 2, 221--288. 



\bibitem[HSV05]{HSVW05}
{\sc Hairer, M.; Stuart, A. M.; Voss, J.; Wiberg, P.} 
\newblock Analysis of SPDEs arising in path sampling. I. The Gaussian case.  
{\em Commun. Math. Sci. } 3  (2005),  no. 4, 587--603.



\bibitem[HSV07]{HSV07}
{\sc Hairer, M.; Stuart, A. M.; Voss, J.}
\newblock  Analysis of SPDEs arising in path sampling. II. The nonlinear case.
{\em  Ann. Appl. Probab. } 17  (2007),  no. 5-6, 1657--1706.


\bibitem[dH00]{dH00}
{\sc den Hollander, F.}
\newblock Large deviations. 
{\em Fields Institute Monographs}, 14. American Mathematical Society, Providence, RI, 2000. x+143 pp. 




\bibitem[Il93]{Il93}
{\sc Ilmanen, T.}
\newblock Convergence of the Allen-Cahn equation to Brakke's motion by mean curvature. 
{\em J. Differential Geom.} 38 (1993), no. 2, 417--461. 



\bibitem[Iw87]{Iw87}
{\sc Iwata, K.}
\newblock An infinite-dimensional stochastic differential equation with state space $C(R)$.
{\em Probab. Theory Related Fields} 74 (1987), no. 1, 141--159. 



\bibitem[Le96]{Le96}
{\sc Ledoux, M.}
\newblock Isoperimetry and Gaussian analysis.
{\em  Lectures on probability theory and statistics (Saint-Flour, 1994)},  165--294, Lecture Notes in Math., 1648, Springer, Berlin, 1996.



\bibitem[OR07]{OR07}
{\sc Otto, F.; Reznikoff, M. G.}
\newblock Slow motion of gradient flows.
{\em J. Differential Equations} 237 (2007), no. 2, 372--420. 


\bibitem[RY99]{RY99}
{\sc Revuz, D.; Yor, M.} 
\newblock Continuous martingales and Brownian motion. Third edition. 
{\em Grundlehren der Mathematischen Wissenschaften} 293. Springer-Verlag, Berlin, 1999. xiv+602 pp.



\bibitem[RV05]{VR05}
{\sc Reznikoff, M. G.; Vanden-Eijnden, E.}
\newblock  Invariant measures of stochastic partial differential equations and conditioned diffusions.  
{\em C. R. Math. Acad. Sci. Paris}  340  (2005),  no. 4, 305--308.



\bibitem[S07]{S07}
{\sc Sheffield, S.}
\newblock Gaussian free fields for mathematicians. 
{\em Probab. Theory Related Fields} 139 (2007), no. 3-4, 521--541. 



\bibitem[S95]{S95}
{\sc Sugiura, M.}
\newblock Metastable behaviors of diffusion processes with small parameter
{\em J. Math. Soc. Japan} 47 (1995), no.4, 755--788. 

\end{thebibliography}
\end{document}